\documentclass[reqno,12pt]{amsart}%
\usepackage{amsfonts}
\usepackage{amsmath}
\usepackage{amssymb}
\usepackage{graphicx}%
\usepackage{enumitem} 
\usepackage{ulem}
\makeatletter
\@addtoreset{equation}{section}
\makeatother

\newtheorem{theorem}{Theorem}[section]

\newtheorem{corollary}[theorem]{Corollary}

\newtheorem{lemma}[theorem]{Lemma}

\newtheorem{proposition}[theorem]{Proposition}
\newtheorem{remark}[theorem]{Remark}

\makeatletter
\def\@makefnmark{}
\makeatother

\begin{document}

\title[Existence of  extremals   for the high order HSM inequalities ]{Existence and symmetry of  extremals for the high order Hardy-Sobolev-Maz'ya inequalities}
\author{Guozhen Lu}
\address{Department of Mathematics\\
University of Connecticut\\
Storrs, CT 06269, USA}
\email{guozhen.lu@uconn.edu}

\author{Chunxia Tao}
\address{School of Mathematics and statistics\\
Beijing Technology And Business University\\
Beijing, 100048, China}
\email{taochunxia01@126.com}

\thanks{The first author was supported by a Simons Collaboration grant  and a Simons Fellowship from the Simons Foundation. The second author was supported by the National Natural Science Foundation of China (No.12201016) and a grant from
Beijing Technology and Business University (QNJJ2022-03).}

\maketitle

\begin{abstract}

In this article, we establish the existence of an extremal function for the  k-th order critical Hardy-Sobolev-Maz'ya (HSM) inequalities on the upper half space $\mathbb{R}^{n+1}_{+}$
when $k\ge 2$ and $n\geq 2k+2$:
 
$$\int_{\mathbb{R}^{n}_{+}}|\nabla^{k}u|^2dx-\prod_{i=1}^{k}\frac{\left(2i-1\right)^2}{4}\int_{\mathbb{R}^{n}_{+}}\frac{u^2}{x_1^{2k}}dx\geq C_{n,k,\frac{2n}{n-2k}} \left(\int_{\mathbb{R}^{n}_{+}}|u|^{\frac{2n}{n-2k}}dx\right)^{\frac{n-2k}{n}}.
$$

The analysis of this extremal problem is challenging due to the presence of the higher order derivatives, the  lack of translation invariance, the inapplicability of rearrangement techniques on the upper half-space, and the presence of a Hardy singularity along the boundary. To overcome these difficulties, instead of directly considering the HSM inequality on the upper half space, we establish the existence of an extremal for its equivalent version: Poincar\'e-Sobolev inequality on the hyperbolic space. We develop a  novel duality theory of the minimizing sequences, the concentration-compactness principle for radial functions in the hyperbolic setting, which combines with the Helgason-Fourier analysis and the Riesz rearrangement inequality on the hyperbolic space, to resolve the lack of compactness issue. As an application, we also obtain the existence of positive symmetric solutions for the high order Brezis-Nirenberg equation on the entire hyperbolic space associated with the GJMS operators $P_k$ (i.e., when $k\ge 2$):
$$
P_{k}\left(f\right)-\alpha f=|f|^{p-2}f
$$
at the critical situation
$\alpha=\prod\limits_{i=1}^{k}\frac{\left(2i-1\right)^2}{4}$ when either $2k+2\leq n$ and $p=\frac{2n}{n-2k}$ or $2k<n$ and $2<p<\frac{2n}{n-2k}$.

\end{abstract}

\section{Introduction}
The classical Hardy-Sobolev-Maz'ya (HSM) inequality provides a refinement of both the  Sobolev and the Hardy inequalities on the upper half space $\mathbb{R}^n_+$. It reads as follows: for $n>2k$, $2<p\leq \frac{2n}{n-2k}$ and $\gamma=\frac{\left(n-2k\right)p}{2}-n$, there exists some positive constant $C_{n,k,p}$ such that for each $u\in C_{c}^{\infty}\left(\mathbb{R}^{n}_{+}\right)$, there holds
\begin{equation}\label{HHSM}
\int_{\mathbb{R}^{n}_{+}}|\nabla^{k}u|^2dx-\prod_{i=1}^{k}\frac{\left(2i-1\right)^2}{4}\int_{\mathbb{R}^{n}_{+}}\frac{u^2}{x_1^{2k}}dx\geq C_{n,k,p} \left(\int_{\mathbb{R}^{n}_{+}}x_1^\gamma|u|^{p}dx\right)^{\frac{2}{p}}.
\end{equation}

The first order case ($k=1$) was established by Maz'ya \cite{Mazya}, while the higher order case was proved by Lu and Yang \cite{LY19}. In particular, when $p=\frac{2n}{n-2k}$ (corresponding to $\gamma=0$), inequality \eqref{HHSM} is referred to as the critical HSM inequality:
\begin{equation}\label{AHSM}
\int_{\mathbb{R}^{n}_{+}}|\nabla^{k}u|^2dx-\prod_{i=1}^{k}\frac{\left(2i-1\right)^2}{4}\int_{\mathbb{R}^{n}_{+}}\frac{u^2}{x_1^{2k}}dx\geq C_{n,k,\frac{2n}{n-2k}} \left(\int_{\mathbb{R}^{n}_{+}}|u|^{\frac{2n}{n-2k}}dx\right)^{\frac{n-2k}{n}}.
\end{equation}

 Lu and Yang \cite{LY19} developed Fourier analysis  techniques on the hyperbolic space which is a noncompact complete Riemannian manifold to establish the HSM inequalities for higher order derivatives on half spaces. (see also Flynn et al \cite{FlynnLuYang} and \cite{LuYang-AIM} for the other cases of symmetric spaces of rank one.)  In fact, they derived the following  Poincar\'e-Sobolev inequalities for the GJMS operators on hyperbolic spaces  $\mathbb{B}^n$: for $n>2k$ and $2<p\leq \frac{2n}{n-2k}$, there exists some constant $C_{n,k,p}$ such that for each $f\in C_{c}^{\infty}\left(\mathbb{B}^n\right)$, there holds

\begin{equation}\label{BPS}
\int_{\mathbb{B}^n}P_k\left(f\right)fdV-\prod_{i=1}^{k}\frac{\left(2i-1\right)^2}{4}\int_{\mathbb{B}^n}|f|^2dV\geq C_{n,k,p}\left(\int_{\mathbb{B}^n}|f|^{p}dV\right)^{\frac{2}{p}},
\end{equation}
where $P_k$ is the GJMS operator on the hyperbolic space  $\mathbb{B}^n$. The precise definition of this operator, along with relevant background on hyperbolic space, is provided in Section \ref{Preliminaries}.
By establishing the equivalence of the Poincar\'e-Sobolev inequalities on $\mathbb{B}^n$ and the HSM inequality on half spaces, they derived the higher order HSM inequality. This approach provides a technically powerful route to the HSM inequality, allowing the higher order HSM inequality to be obtained in a unified and indirect manner via its hyperbolic counterpart.
Motivated by this equivalence, in this article,  we first study the existence of the extremal functions for the Poincar\'e-Sobolev inequalities on the hyperbolic space   $\mathbb{B}^n$ and use this framework to establish the existence of extremal functions for the corresponding HSM inequalities.
\medskip

The $k$-th order Sobolev inequality on the hyperbolic space $\mathbb{B}^n$ reads as follows:
\begin{equation}\label{BSobolev}
\int_{\mathbb{B}^n}P_k\left(f\right)fdV\geq S_{n,k}\left(\int_{\mathbb{B}^n}|f|^{\frac{2n}{n-2k}}dV\right)^{\frac{n-2k}{n}},
\end{equation}
where $f\in C_{c}^{\infty}\left(\mathbb{B}^n\right)$, $1\leq k <\frac{n}{2}$, $S_{n,k}$ is the sharp constant of $k$-th order Sobolev  inequality on $\mathbb{R}^n$.
(see Hebey \cite{Hebey} and Liu \cite{Liu}). In \cite{LY19}, Lu and Yang  gave an alternative way to prove \eqref{BSobolev}. In fact, let $g(x)=\big(\frac{2}{1-|x|^2}\big)^{\frac{n}{2}-k}f(x)$, using the conformal transformation laws (see Lemma 5.1 of \cite{LY19}) $$P_k(f)=\big(\frac{2}{1-|x|^2}\big)^{-\frac{n}{2}-k}(-\Delta)^k\big((\frac{2}{1-|x|^2})^{\frac{n}{2}-k}f(x)\big),$$
we can see that the $k$-th order Sobolev inequality on the hyperbolic space $\mathbb{B}^n$ is equivalent to the following $k$-th order  Sobolev inequality on the unit ball of $\mathbb{R}^n$:
\begin{equation}\label{BallSobolev}
\int_{\mathbb{B}^n}(-\Delta)^k g\cdot gdx\geq S_{n,k}\left(\int_{\mathbb{B}^n}|g|^{\frac{2n}{n-2k}}dx\right)^{\frac{n-2k}{n}}.
\end{equation}
Obviously, the $k$-th order Sobolev inequality on the unit ball of $\mathbb{R}^n$ has no extremals using translation, dilation and classification of extremals of Sobolev inequality in $\mathbb{R}^n$ (see Talenti \cite{Talenti} and Aubin \cite{Aubin} and also Proposition 1.43 of \cite{Willem} for the first order, and  Lieb \cite{Lieb} forthe  high-order case). Hence the $k$-th order Sobolev inequality on the hyperbolic space $\mathbb{B}^n$ does not admit any extremal function.
\medskip

On the other hand, we have the following $k$-th Poincar\'e inequality on the hyperbolic space $\mathbb{B}^n$:
\begin{equation}\label{Bpoincare}
\int_{\mathbb{B}^n}P_k\left(f\right)fdV\geq \prod_{i=1}^{k}\frac{\left(2i-1\right)^2}{4}\left(\int_{\mathbb{B}^n}|f|^{2}dV\right),
\end{equation}
where $f\in C_{c}^{\infty}\left(\mathbb{B}^n\right)$, and the constant $\prod_{i=1}^{k}\frac{\left(2i-1\right)^2}{4}$ is sharp. We need to note that the sharp constants cannot be attained.    Combining $k$-th order Sobolev inequality and $k$-th order Poincar\'e inequality on the hyperbolic space $\mathbb{B}^n$, there must exist some refinement of both of these two kinds of sharp inequalities. The refinement is simply demonstrated by the high-order Poincare-Sobolev inequality on the hyperbolic space \eqref{BPS}. However, the problem for the existence of extremal of high-order Poincare-Sobolev inequality on the hyperbolic space still remains open.

For the first order critical HSM inequality, Hebey \cite{Hebey} showed that the best constant $C_{n,1,\frac{2n}{n-2}}$ is strictly smaller than the Sobolev constant $S_{n,1}$ for $n\geq 4$. Tertikas and Tintarev  \cite{TT} further proved that the existence of the extremal function for $n\geq 4$. On the other hand, Benguria, Frank and Loss \cite{BFL} demonstrated that the best constant $C_{n,1,\frac{2n}{n-2}}$ is the same as the Sobolev constant $S_{n,1}$ and the best constant $C_{n,1,\frac{2n}{n-2}}$ is not achieved when  $n=3$. For the high order critical case, Lu and Yang    proved in \cite {LY22}  that the best constant $C_{n,k,\frac{2n}{n-k}}$ for the kth-order HSM inequality coincides with the $k$-th Sobolev contant $S_{n,k}$ if $n=2k+1$ and the best constant $C_{n,k,\frac{2n}{n-k}}$ is strictly smaller than $k$-th Sobolev contant $S_{n,k}$ if $n\geq 2k+2$. However, the attainability of the best constant $C_{n,k,\frac{2n}{n-2k}}$ remains open for $n\ge 2k+2$ when $k\ge 2$. In \cite{TT}, the authors employed a clever transformation that reduces the HSM inequality to a weighted Sobolev inequality on the upper half-space, and by applying the concentration compactness principle, they established the existence of extremal functions in the first order case $k=1$. Due to the increased complexity of higher order operators, such a transformation is no longer available in the higher order setting.
\medskip

Our first main result concerns the existence of extremal functions for the critical higher order Poincar\'e-Sobolev inequality.
We define the Hilbert space $H\left(\mathbb{B}^n\right)$ by the completion of $C_c^{\infty}\left(\mathbb{B}^n\right)$ under the norm
$$\|f\|_{H\left(\mathbb{B}^n\right)}=\left(\int_{\mathbb{B}^n}P_k\left(f\right)fdV-\prod_{i=1}^{k}\frac{\left(2i-1\right)^2}{4}\int_{\mathbb{B}^n}|f|^2dV\right)^{\frac{1}{2}}.$$

It is noted that $\|f\|_{H\left(\mathbb{B}^n\right)}$ is indeed a norm by using the Poincare-Sobolev inequality.

\begin{theorem}\label{EPS} Assume that $p=\frac{2n}{n-2k}$ and $n\geq2k+2$.
    Then there is a positive function $f_0\in H\left(\mathbb{B}^n\right)$ which achieves the equality 
    for  the inequality \eqref{BPS}. Furthermore, the extremal function $f_0$ must be radially symmetric and monotone decreasing about some point $P\in \mathbb{B}^n$, that is, $f_0$ is a constant on the geodesic sphere centered at $P\in \mathbb{B}^n$ and radially decreasing about the geodesic distance from P. 
\end{theorem}

The proof relies on developing the Lions' type of the concentration-compactness principle \cite{Lions841, Lions1} for differential inequalities on the hyperbolic space. The concentration-compactness principle for differential inequalities can be divided into the first and the second differential concentration-compactness. One of the key points is to prove the minimizing sequence $\{u_k\}$ for higher order Poincar\'e-Sobolev inequality is tight. In order to achieve this, we must exclude the vanishing  and dichotomy phenomenon. The loss of invariance of dilation causes much challenge to exclude the vanishing phenomenon. In order to overcome these difficulties, we need to find the radially minimizing sequence which decays at infinity of hyperbolic space and help to exclude the vanishing phenomenon. However, P\"{o}lya-Szeg\"{o} rearrangement inequality is invalid for high-order differential inequality. Applying the equivalence of Poincar\'e-Sobolev inequality and Hardy-Littlewood-Sobolev inequality on the hyperbolic space, Riesz's rearrangement inequality on hyperbolic space and Green's representative formula, one can find the radially minimizing sequence for Poincar\'e-Sobolev inequality. To exclude  the dichotomy, one cannot follow the same line of Lions' proof in \cite{Lions1}. In fact, recall the proof of excluding dichotomy for minimizing sequence $w_k$ of $L^p$ Sobolev inequality involving high order derivatives: $$I=\inf\{\int_{\mathbb{R}^n}|\nabla^{m}w|^pdx,\ \ \int_{\mathbb{R}^n}|w|^{\frac{np}{n-mp}}dx=1\}.$$ Lions constructed a suitable concentration-function $$\rho_k=\sum_{j=0}^{m}|D^jw_k|^{q_j},\  q_j=\frac{np}{n-(m-j)p}$$
and prove that for any $\epsilon>0$, there exists $y_k\in \mathbb{R}^n$ and $R_{\epsilon}>0$ such that $\int_{|x-y_k|\geq R}\rho_k(x)dx<\epsilon$ for any $R>R_{\epsilon}$. This is achieved by scaling cut-off function and 
$$\int_{R_0<|x-y_k|<R_k}\sum_{j=0}^{m-1}|D^jw_k|^{q_j}dx=o_k(1).$$ However, for higher order Poincar\'e-Sobolev inequality in hyperbolic space, it is impossible to introduce the replacement of concentration-function $\rho_k$ of $\mathbb{R}^n$. Moreover, the loss of scaling cut-off function can also add extra difficulty. In order to overcome this difficulty, we found a quantitative relationship between the minimizing sequences for the Poincare-Sobolev inequalities and those for the  equivalent integral inequalities (It is differen from the equivalence of extremals), allowing us to apply the concentration-compactness techniques for equivalent integral inequalities to rule out dichotomy of the minimizing sequences for Poincar\'e-Sobolev inequality. Once we prove that the minimizing sequence is compact, we can follow the standard second Lions' concentration-compactness procedure to exclude the minimizing sequence converging $\delta$-measure using $C_{n,k,\frac{2n}{n-k}}<S_{n,k}$.

\medskip

Through the equivalence, combining Theorem \ref{EPS}, we immediately obtain the existence of the extremal function for the  high order critical Hardy-Sobolev-Maz'ya inequality \eqref{AHSM} on the upper half space.

\begin{theorem}\label{EHMS}
 Assume $n\geq2k+2$.  Then there is a positive function $f_0\in \widetilde{H}\left(\mathbb{R}^{n}_{+}\right)$ which achieves the equality in inequality
 \eqref{AHSM},
where $\widetilde{H}\left(\mathbb{R}^{n}_{+}\right)$ is the completion of $C_c^{\infty}\left(\mathbb{R}^{n}_{+}\right)$ under the norm
$$\|u\|_{\widetilde{H}\left(\mathbb{R}^{n}_{+}\right)}=\left(\int_{\mathbb{R}^{n}_{+}}|\nabla^{k}u|^2dx-\prod_{i=1}^{k}\frac{\left(2i-1\right)^2}{4}\int_{\mathbb{R}^{n}_{+}}\frac{u^2}{x_1^{2k}}dx\right)^{\frac{1}{2}}.$$

\end{theorem}

For $n=2k+1$, according to Lemma 4.3 of \cite{LY22}, we can deduce the nonexistence of the extremal function for critical inequality \eqref{BPS} and \eqref{HHSM}. Different from the critical case, the subcritical inequalities have an extremal function for all $n>2k$. We have the following results:

\begin{theorem}\label{E-subHPS}
  Assume that $2<p<\frac{2n}{n-2k}$ and $n>2k$. Then there is a positive function $f_0\in H\left(\mathbb{B}^n\right)$ which achieves the equality in the inequality \eqref{BPS},
\end{theorem}

\begin{theorem}\label{E-subHHSM}
Assume that $2<p<\frac{2n}{n-2k}$ and $n>2k$. 
 Then there is a positive function $f_0\in \widetilde{H}\left(\mathbb{R}^{n}_{+}\right)$ which achieves the equality in   inequality \eqref{HHSM}.
 
\end{theorem}

The Poincar\'e-Sobolev inequalities  are highly linked with the  Brezis-Nirenberg problem \cite{Brezis} on the entire hyperbolic space $\mathbb{B}^n$:
\begin{equation}\label{sublambdaBN}
P_{k}\left(f\right)-\alpha f=|f|^{p-2}f.
\end{equation}
For $k=1$, Mancini and Sandeep \cite{MS} showed that the entire solution exists either when  $n\geq4$, $p=\frac{2n}{n-2}$, $0<\alpha\leq\frac{1}{4}$ or when $n\geq3$, $1<p<\frac{2n}{n-2}$ and $\alpha\leq\frac{1}{4}$. Bhakta et al. \cite{BGKM25} proved the quantitative stability for the equation \eqref{sublambdaBN} when $k=1$. 
 For $k\geq 2$, Li et al. \cite{LLY} established the existence of the nontrivial solution  when $n\geq 2k+2$, $p=\frac{2n}{n-2k}$, $\alpha<\prod\limits_{i=1}^{k}\frac{\left(2i-1\right)^2}{4}$. It is not hard to calculate that the corresponding Euler-Lagrange equation of the inequality \eqref{BPS} is the  equation \eqref{sublambdaBN} with $\alpha=\prod\limits_{i=1}^{k}\frac{\left(2i-1\right)^2}{4}$, $2<p\leq\frac{2n}{n-2k}$.  A straightforward corollary of Theorem \ref{EPS} and Theorem \ref{E-subHPS} is as follows:

\begin{corollary}\label{EL}
Let $\alpha=\prod\limits_{i=1}^{k}\frac{\left(2i-1\right)^2}{4}$. If either $2k+2\leq n$, $p=\frac{2n}{n-2k}$ or $2k<n$, $2<p<\frac{2n}{n-2k}$,  then there exists at least one positive solution u in $H\left(\mathbb{B}^n\right)$ to the equation \eqref{sublambdaBN},  which is radially symmetry  and monotone decreasing about some point $P\in \mathbb{B}^n$.
\end{corollary}

\begin{remark}
   It is the first time to derive the existence of positive solutions for high-order Brezis-Nirenberg problem for $\alpha=\prod\limits_{i=1}^{k}\frac{\left(2i-1\right)^2}{4}$ on the whole hyperbolic space.
\end{remark}

\section{Preliminaries on the Hyperbolic space}\label{Preliminaries}

In this section, we recall some facts on the hyperbolic space. There are several models of hyperbolic space, such as  the Poincar\'e half space model and the Poincar\'e ball model. We denote by $\mathbb{B}^n$ the Poincar\'e ball model. In this article, we use the Poincar\'e ball model throughout.
\subsection{The Poincar\'e ball model $\mathbb{B}^n$}
The Poincar\'e ball model $\mathbb{B}^n$ is the unit ball $$\mathbb{B}^n=\{x=(x_1,...,x_n)\in\mathbb{R}^{n}||x|<1\}$$ 
equipped with the usual Poincar\'e metric $g=\left(\frac{2}{1-|x|^2}\right)^2g_e$, where $g_e$ represents
the standard Euclidean metric. The hyperbolic volume element is  $dV=\left(\frac{2}{1-|x|^2}\right)^ndx$. For measurable set $E\in \mathbb{B}^n$, we use $\tau\left(E\right)$ to denote the hyperbolic measure of $E$, that is,
$$\tau\left(E\right)=\int_{E}dV=\int_{E}\left(\frac{2}{1-|x|^2}\right)^ndx.$$
For $0<R<1$, we denote by $B^n\left(0,R\right)$ the ball  centered at origin with Euclidean radius $R$, then 
$$\tau\left(B^n\left(0,R\right)\right)=\int_{B^n\left(0,R\right)}dV=\omega_{n-1}\int_{0}^{R}\left(\frac{2}{1-|s|^2}\right)^ns^{n-1}ds,$$
where $\omega_{n-1}$ is the surface area of the Euclidean unit sphere $\mathbb{S}^{n-1}$. For $x\in \mathbb{B}^n$, $r\in(0,+\infty)$, we denote by  $B_{\mathbb{H}}\left(x,r\right)$ the hyperbolic ball centered at $x$ with hyperbolic radius $r$. In fact, 
$$B_{\mathbb{H}}\left(x,r\right)=\{y\in B^n(0,1)\ |\  \rho(y,x)=\log\big(\frac{1+|T_x(y)|}{1-|T_x(y)|}\big)<r\},$$
where $T_x(y)$ is the M\"{o}buis transform to be defined in next subsection. 
It is easy to check that 
hyperbolic geodesic ball $B_{\mathbb{H}}(0,r)$ is just Euclidean ball $B^n(0, \tanh(\frac{r}{2}))$, where $\tanh(t)=\frac{e^{t}-e^{-t}}{e^{t}+e^{-t}}<1$.

The hyperbolic gradient $\nabla_{\mathbb{H}}$ and the associated Laplace-Beltrami operator $\Delta_{\mathbb{H}}$  are given respectively by
$$\nabla_{\mathbb{H}}=\frac{\left(1-|x|^2\right)}{2}\nabla,\ \ \Delta_{\mathbb{H}}=\frac{1-|x|^2}{4}\left(\left(1-|x|^2\right)\Delta+2\left(n-2\right)\sum_{i=1}^{n}x_i\frac{\partial}{\partial x_i}\right),$$
where $\Delta=\sum\limits_{i=1}^{n}\frac{\partial^2}{\partial x_i^2}$ is the Laplace operator on $\mathbb{R}^n$. The GJMS operator on the hyperbolic space $\mathbb{B}^n$ is defined as (see  \cite{FG1, FG2, GJMS, Juhl})
$$P_k=P_1\left(P_1+2\right)\cdot\cdot\cdot\cdot\cdot\left(P_1+k\left(k-1\right)\right),\ k\in \mathbb{N}^+,$$
where $P_1=-\Delta_{\mathbb{H}}-\frac{n\left(n-2\right)}{4}$ is the conformal Laplacian on hyperbolic space $\mathbb{B}^n$.
\medskip

\subsection{M\"{o}bius transformations}
For each $a\in \mathbb{B}^n$, the M\"{o}bius transformations $T_a$ is defined as (see e.g. \cite{Ahlfors, Hua})
$$T_a\left(x\right)=\frac{|x-a|^2a-\left(1-|a|^2\right)\left(x-a\right)}{1-2x\cdot a+|x|^2a^2},$$
where $x\cdot a$ denotes the scalar product in $\mathbb{R}^n$.  We need to note that the measure on $\mathbb{B}^n$ is invariant under the M\"{o}bius transformations. Using the M\"{o}bius transformations, we can define the distance from $x$ to $y$ on $\mathbb{B}^n$  by
$$\rho\left(x,y\right)=\rho\left(T_{x}\left(y\right)\right)=\rho\left(T_{y}\left(x\right)\right)=\log \frac{1+|T_{y}\left(x\right)|}{1-|T_{y}\left(x\right)|}.$$
The distance from the origin to $x\in \mathbb{B}^n$ is $\rho\left(x\right)=\log \frac{1+|x|}{1-|x|}$. 
 The convolution of measurable functions $f$ and $g$ on $\mathbb{B}^n$ can be defined as (see e.g. \cite{LP})
$$\left(f\ast g\right)\left(x\right)=\int_{\mathbb{B}^n}f\left(y\right)g\left(T_x\left(y\right)\right)dV_y.$$

\subsection{Helgason-Fourier transform}
We refer the reader to \cite{Helgason84, Helgason08, Thangavelu02, Thangavelu04} for the Helgason-Fourier transform on the hyperbolic space. 
Set $$e_{\lambda,\xi}\left(x\right)=\left(\frac{\sqrt{1-|x|^2}}{|x-\xi|}\right)^{n-1+i\lambda},\ x\in \mathbb{B}^n,\ \lambda\in \mathbb{R}, \ \xi \in \mathbb{S}^{n-1}.$$
In terms of ball model, the Fourier transform of a function $f$ on $\mathbb{B}^n$ can be difinded as 
$$\widehat{f}\left(\lambda,\xi\right)=\int_{\mathbb{B}^n}f\left(x\right)e_{-\lambda, \xi}\left(x\right)dV$$ 
provided this integral exists. If $g\in C_c^\infty(\mathbb{B}^n)$ is radial, then $\widehat{\left(f\ast g\right)}=\widehat{f}\cdot\widehat{g}$.
Moreover, for $f\in C_c^\infty(\mathbb{B}^n)$, we have the following inversion formula:
$$f\left(x\right)=D_{n}\int_{-\infty}^{\infty}\int_{\mathbb{S}^{n-1}}\widehat{f}\left(\lambda,\xi\right)e_{\lambda,\xi}\left(x\right)|c\left(\lambda\right)|^{-2}d\lambda d\xi,$$
where $D_n=\frac{1}{2^{3-n}\pi |\mathbb{S}^{n-1}|}$ and $c\left(\lambda\right)$ is the Harish-Chandra function given by (see e.g.\cite{LP})
$$c\left(\lambda\right)=\frac{2^{n-1-i\lambda}\Gamma\left(n/2\right)\Gamma\left(i\lambda\right)}{\Gamma\left(\frac{n-1+i\lambda}{2}\right)\Gamma\left(\frac{1+i\lambda}{2}\right)}.$$
For $g\in L^{2}\left(\mathbb{B}^n\right)$, $h\in L^{2}\left(\mathbb{B}^n\right)$, we also have the following Plancherel formula:
$$\int_{\mathbb{B}^n}g\left(x\right)h\left(x\right)dV=D_{n}\int_{-\infty}^{\infty}\int_{\mathbb{S}^{n-1}}\widehat{g}\left(\lambda,\xi\right)\widehat{h}\left(\lambda,\xi\right)|c\left(\lambda\right)|^{-2}d\lambda d\xi.$$
Since $e_{\lambda,\xi}$ is an eigenfunction of $-\Delta_{\mathbb{H}}$ with eigenvalue $\frac{\left(n-1\right)^2+\lambda^2}{4}$, then for $f\in C^{\infty}_c\left(\mathbb{B}^n\right)$, there holds $\widehat{-\Delta_{\mathbb{H}} f}\left(\lambda, \xi\right)=\frac{\left(n-1\right)^2+\lambda^2}{4} \widehat{f}\left(\lambda, \xi\right)$. According to the definition of GJMS operator, we can check that
$$\widehat{P_k\left(f\right)}\left(\lambda,\xi\right)=\frac{\lambda^2+1}{4}\left(\frac{\lambda^2+1}{4}+2\right)\cdot\cdot\cdot\cdot\cdot\left(\frac{\lambda^2+1}{4}+k\left(k-1\right)\right)\widehat{f}\left(\lambda, \xi\right).$$
Similar to the Euclidean setting, the fractional Laplacian on hyperbolic space is defined as follows:
$$\widehat{\left(-\Delta_{\mathbb{H}^n}\right)^{\gamma}f}\left(\lambda,\xi\right)=\left(\frac{\left(n-1\right)^2+\lambda^2}{4}\right)^{\gamma} \widehat{f}\left(\lambda, \xi\right), \gamma\in \mathbb{R}.$$

\subsection{Rearrangement on hyperbolic space}
Now, we recall some facts about the rearrangement on hyperbolic space (see \cite{Albert}). For a measurable function $f$ on $\mathbb{B}^n$, the one-dimensional rearrangement $f^{\sharp}\left(s\right)$: $\mathbb{R^+}\to \mathbb{R}$   is defined as:
$$f^{\sharp}\left(s\right)=\inf\{t>0:\ \int_{\{x\in \mathbb{B}^n: |f|> t\}}dV\leq s\};$$ and
the hyperbolic geodesically decreasing rearrangement $f^{*}\left(x\right)$: $\mathbb{B}^n \to \mathbb{R}$ of $f$ is given by
 $$f^{*}\left(x\right)=f^{\sharp}\left(\tau\left(B_{\mathbb{H}}\left(0,\rho\left(x\right)\right)\right)\right).$$
We need to note that $f^{*}\left(x\right)$ is radially decreasing about origin since $\rho\left(x\right)=\log \left(\frac{1+|x|}{1-|x|}\right)$ is decreasing about origin.
Through the rearrangement property, we can know that $$\tau\{x\in \mathbb{B}^n: f^{*}\left(x\right)>t\}=\tau{\left(\{x\in \mathbb{B}^n: |f|\geq t\}\right)},$$
and
 $$\int_{\mathbb{B}^n}|f^{*}|^pdV=\int_{\mathbb{B}^n}|f|^pdV,\   0<p<\infty.$$
\medskip

\section{The proof of Theorem \ref{EPS}}
In this section, we shall prove the existence of the extremal function for the high order critical Poincar\'e-Sobolev inequality on $\mathbb{B}^n$ ( i.e. Theorem \ref{EPS}). That is, we consider the attainability of the following minimizing problem:
\begin{equation}\label{minimizing}
C_{n,k,\frac{2n}{n-2k}}:=\inf\{\|f\|^2_{H(\mathbb{B}^n)}: \int_{\mathbb{B}^n}|f|^{\frac{2n}{n-2k}}dV=1\}.
\end{equation}
The proof is divided into three steps,  each formulated as a separate proposition. In subsection \ref{subsecradiallysequence}, we shall pick up radially  decreasing minimizing sequence for this minimizing  problem (see Proposition \ref{proradiallysequence}).  Through the Helgason-Fourier analysis and duality theory, we transform the Poincar\'e-Sobolev inequality into the corresponding Hardy-Littlewood-Sobolev inequality on the hyperbolic space. Through establishing the relationship of the extremizing sequence between the Poincar\'e-Sobolev inequality and the  Hardy-Littlewood-Sobolev inequality (see Lemma \ref{relationshipsequence}), together with  proving the monotone decreasing property of the kernel (see Lemma \ref{kernel-decreasing}) and applying the Symmetrization Lemma, we can pick up radially  decreasing minimizing sequence. In subsection \ref{subsectight},  we shall show the tightness for the radially  decreasing minimizing sequence (see Proposition \ref{protight}). According to the first Lions' concentration compactness Lemma (see Lemma \ref{kernel-decreasing}), there are three phenomenons occurs. Through establishing the  radial estimates at infinity of hyperbolic space   (see Lemma \ref{radial}) for the minimizing sequence, we can rule out vanishing phenomenon. By applying again the quantitative relationship between the minimizing sequences (Lemma \ref{relationshipsequence}), we prove that the dichotomy phenomenon cannot occur for the minimizing sequence of the corresponding Hardy–Littlewood–Sobolev inequality, rather than excluding it directly for the Poincar\'e-Sobolev minimizing sequence. In subsection \ref{subconvergence}, we shall prove that the tight and radially  decreasing minimizing sequence strongly converges to a non-zero function in $H(\mathbb{B}^n)$ (see Proposition \ref{proconvergence}), thereby completing the proof of  Theorem \ref{EPS}. 
We  first show that $H\left(\mathbb{B}^n\right)\subset W^{k,2}_{loc}\left(\mathbb{B}^n\right).$ And then we establish the 
second Lions' concentration compactness Lemma on the hyperbolic space $\mathbb{B}^n$ for the high order critical Poincar\'e-Sobolev inequality (see Lemma \ref{con2}). Together with the fact that the Poincar\'e-Sobolev inequality constant $C_{n,k}$ is strictly smaller than the Sobolev constant $S_{n,k}$, we obtain the strong convergence.

\subsection{Picking up radially  decreasing minimizing sequence}\label{subsecradiallysequence}

\begin{proposition}\label{proradiallysequence}
There exists a radially  decreasing minimizing sequence $\{f_m\}$ for minimizing  problem \eqref{minimizing}. Specifically,  there exists a sequence $\{f_m\}$ such that 
$$\lim_{m\to \infty}\|f_m\|^2_{H(\mathbb{B}^n)}=C_{n,k,\frac{2n}{n-2k}},\ \ \ \int_{\mathbb{B}^n}|f_m|^{\frac{2n}{n-2k}}dV=1$$
and for each $m\in \mathbb{N}$, the function $f_m(x)$ is  radially decreasing about the geodesic distance from the origin.
\end{proposition}

\begin{proof}
\medskip
For $f\in C_{c}^{\infty}\left(\mathbb{B}^n\right)$, define $G_k\left(f\right)$ to be given by
 $$G_k\left(f\right)=P_k\left(f\right)-\prod_{i=1}^{k}\frac{\left(2i-1\right)^2}{4}f,$$
 then $$\widehat{G_k\left(f\right)}=\left(\frac{\lambda^2+1}{4}\cdots\left(\frac{\lambda^2+1}{4}+k\left(k-1\right)\right)-\prod_{i=1}^{k}\frac{\left(2i-1\right)^2}{4}\right)\widehat{f}\left(\lambda, \xi\right).$$

In analogy with fractional Laplacian setting, the fraction operator $G_{k}^{\gamma}$  on $f$ is defined as follows:
\begin{equation}\begin{split}
\widehat{G_{k}^{\gamma}\left(f\right)}\left(\lambda, \xi\right)=&\left(\frac{\lambda^2+1}{4}\cdots\left(\frac{\lambda^2+1}{4}+k\left(k-1\right)\right)-\prod_{i=1}^{k}\frac{\left(2i-1\right)^2}{4}\right)^{\gamma}\widehat{f}\left(\lambda, \xi\right), \gamma\in \mathbb{R}.
\end{split}\end{equation}
Applying the Plancherel formula, we derive that
{\small
\begin{equation}\begin{split}
&\int_{\mathbb{B}^n}P_k\left(f\right)fdV-\prod_{i=1}^{k}\frac{\left(2i-1\right)^2}{4}\int_{\mathbb{B}^n}|f|^2dV\\
&\ \ =D_{n}\int_{-\infty}^{\infty}\int_{\mathbb{S}^{n-1}}\widehat{G_{k}\left(f\right)}\left(\lambda,\xi\right)\widehat{f}\left(\lambda, \xi\right)|c\left(\lambda\right)|^{-2}d\lambda d\xi\\
&\ \ =D_{n}\int_{-\infty}^{\infty}\int_{\mathbb{S}^{n-1}}\left(\frac{\lambda^2+1}{4}\cdots\left(\frac{\lambda^2+1}{4}+k\left(k-1\right)\right)-\prod_{i=1}^{k}\frac{\left(2i-1\right)^2}{4}\right)\cdot|\widehat{f}\left(\lambda, \xi\right)|^2|c\left(\lambda\right)|^{-2}d\lambda d\xi\\
&\ \ =D_{n}\int_{-\infty}^{\infty}\int_{\mathbb{S}^{n-1}}\left(\left(\frac{\lambda^2+1}{4}\cdots\left(\frac{\lambda^2+1}{4}+k\left(k-1\right)\right)-\prod_{i=1}^{k}\frac{\left(2i-1\right)^2}{4}\right)^{\frac{1}{2}}\widehat{f}\left(\lambda,\xi\right)\right)^2\cdot |c\left(\lambda\right)|^{-2}d\lambda d\xi\\
&\ \ =D_{n}\int_{-\infty}^{\infty}\int_{\mathbb{S}^{n-1}}|\widehat{G_{k}^{\frac{1}{2}}\left(f\right)}\left(\lambda,\xi\right)|^2|c\left(\lambda\right)|^{-2}d\lambda d\xi\\
&\ \ =\int_{\mathbb{B}^n}|G_{k}^{\frac{1}{2}}\left(f\right)|^2dV.
\end{split}\end{equation}}

Hence, the critical Poincar\'e-Sobolev inequality on  $\mathbb{B}^n$ can be written as
\begin{equation}\label{G1}
\int_{\mathbb{B}^n}|G_{k}^{\frac{1}{2}}\left(f\right)|^2dV\geq C_{n,k,\frac{2n}{n-2k}}\left(\int_{\mathbb{B}^n}|f|^{\frac{2n}{n-2k}}dV\right)^{\frac{n-2k}{n}}.
\end{equation}

Let $G_{k}^{\frac{1}{2}}\left(f\right)=g$, then $f=G_{k}^{-\frac{1}{2}}\left(g\right)$. By the inversion of Fourier transform, $f$ is equal to
{\small
\begin{equation*}
D_{n}\int_{-\infty}^{\infty}\int_{\mathbb{S}^{n-1}}\left(\frac{\lambda^2+1}{4}\cdots\left(\frac{\lambda^2+1}{4}+k\left(k-1\right)\right)-\prod_{i=1}^{k}\frac{\left(2i-1\right)^2}{4}\right)^{-\frac{1}{2}}\widehat{g}\left(\lambda, \xi\right)e_{\lambda,\xi}\left(x\right)|c\left(\lambda\right)|^{-2}d\lambda d\xi.
\end{equation*}
}

Then the above Poincar\'e-Sobolev inequality \eqref{G1} reduces to
\begin{equation}\label{reducepoincare}
\int_{\mathbb{B}^n}|g|^2dV\geq C_{n,k,\frac{2n}{n-2k}}\left(\int_{\mathbb{B}^n}|G_k^{-\frac{1}{2}}\left(g\right)|^{\frac{2n}{n-2k}}dV\right)^{\frac{n-2k}{n}}.
\end{equation}
By duality, the above inequality is in fact equivalent to the following inequality on hyperbolic space $\mathbb{B}^n$:
\begin{equation}\label{G2}
\left(\int_{\mathbb{B}^n}|h|^{\frac{2n}{n+2k}}dV\right)^{\frac{n+2k}{n}}\geq C_{n,k,\frac{2n}{n-2k}}\int_{\mathbb{B}^n}|G_k^{-\frac{1}{2}}\left(h\right)|^{2}dV.
\end{equation}
We denote by $G_k^{-1}\left(x,y\right)$ the kernel function of the operator $G_k^{-1}$.
Then inequality \eqref{G2} is equivalent to the following Hardy-Littlewood-Sobolev inequalities:
\begin{equation}\label{BHLS}
C_{n,k,\frac{2n}{n-2k}}\int_{\mathbb{B}^n}\int_{\mathbb{B}^n}h\left(x\right)G_k^{-1}\left(x,y\right)h\left(y\right)dV_{x}dV_{y}\leq \left(\int_{\mathbb{B}^n}|h|^{\frac{2n}{n+2k}}dV\right)^{\frac{n+2k}{n}}.
\end{equation}

From  Lemma \ref{kernel-decreasing}  (presented later), we know that  $G^{-1}_k\left(x,y\right)$ is symmetric and decreasing about the distance function $\rho\left(T_{x}\left(y\right)\right)$. From the Symmetrization Lemma \ref{becknersymm} in \cite{Beckner93}, we can pick up radially decreasing maximizing sequence  $\{h_m\}$ for the Hardy-Littlewood-Sobolev inequality \eqref{BHLS}. By the direct equivalence of the inequality \eqref{BHLS} and \eqref{G2}, we know that  $\{h_m\}$ is also the radially decreasing maximizing sequence for inequality \eqref{G2}, that is, 
$$\int_{\mathbb{B}^n}|h_m|^{\frac{2n}{n+2k}}dV=1,\ \ \lim\limits_{m\rightarrow +\infty} \int_{\mathbb{B}^n}|G_k^{-\frac{1}{2}}\left(h_m\right)|^{2}dV=C_{n,k,\frac{2n}{n-2k}}^{-1}.$$ 
In the following Lemma \ref{relationshipsequence}, we proved the relationship between the extremizing sequence of \eqref{G1} and the extremizing sequence of \eqref{G2}. Pick $f_m=G_{k}^{-1}\left(h_m\right)$, then $f_m\left(x\right)$ is a radially decreasing extremizing sequence for inequality \eqref{G1}. Up to some constants, we have 
$$\lim_{m\rightarrow\infty}\|f_m\|^2_{H(\mathbb{B}^n)}=C_{n,k,\frac{2n}{n-2k}},\ \ \ \int_{\mathbb{B}^n}|f_m|^{\frac{2n}{n-2k}}dV=1.$$
\end{proof}

In the proof of Proposition \ref{proradiallysequence} above, we rely on the following three lemmas. Lemma \ref{becknersymm} is taken from \cite{Beckner93}; a detailed proof can also be found in \cite{BT76, Beckner92}. Lemma \ref{relationshipsequence} establishes an explicit relationship between the dual inequalities \eqref{G1} and \eqref{G2}. To establish the radial monotonicity in Lemma \ref{kernel-decreasing}, we adapt the approach introduced in \cite{LLY}.

\begin{lemma}{\cite[Symmetrization Lemma]{Beckner93}}\label{becknersymm}
Let $K$ be a monotonically decreasing function on $\mathbb{R}$, then
$$\int_{\mathbb{B}^n}\int_{\mathbb{B}^n}h\left(x\right)K(\rho\left(T_{x}\left(y\right)\right))h\left(y\right)dV_{x}dV_{y}\leq \int_{\mathbb{B}^n}\int_{\mathbb{B}^n}h^{*}\left(x\right)K(\rho\left(T_{x}\left(y\right)\right))\left(x,y\right)h^{*}\left(y\right)dV_{x}dV_{y},$$
where $h$  is nonnegative measurable functions with $h^{*}$ denoting its hyperbolic geodesically decreasing rearrangement on $\mathbb{B}^n$ and the integrand on the left side is integrable. If $K$ is strictly decreasing, then the above inequality is strict unless $h(x)=h^*(T_a(x))$ for any $a\in\mathbb{B}^n$.
\end{lemma}

\begin{lemma}\label{relationshipsequence}
Let $\{f_m\}$ be the extremizing sequence for inequality \eqref{G1} satisfying 
\begin{equation}\label{dualsequece1}
    \|G_k^{\frac{1}{2}}(f_m)\|_{L^2(\mathbb{B}^n)}=1\ \text{and}\ \ \lim_{m\to\infty}\|f_m\|_{L^{\frac{2n}{n-2k}}(\mathbb{B}^n)}=C^{-1}_{n,k,\frac{2n}{n-2k}}.
\end{equation}
Define $h_m=(f_m)^{\frac{n+2k}{n-2k}}$. Then $\{h_m\}$ is an extremizing sequence for inequality \eqref{G2}; precisely,
$$\lim_{m\to\infty}\frac{\|G_k^{-\frac{1}{2}}(h_m)\|_{L^{2}(\mathbb{B}^n)}}{\|h_m\|_{L^{\frac{2n}{n+2k}}(\mathbb{B}^n)}}=C^{-1}_{n,k,\frac{2n}{n-2k}}.$$
On the other hand, let $\{h_m\}$ be the extremizing sequence for inequality \eqref{G2} satisfying 
\begin{equation}\label{dualsequece2}
\|h_m\|_{L^{\frac{2n}{n+2k}}(\mathbb{B}^n)}=1 \ \text{and} \ \lim_{m\to\infty}\|G_k^{-\frac{1}{2}}(h_m)\|_{L^{2}(\mathbb{B}^n)}=C^{-1}_{n,k,\frac{2n}{n-2k}}.
\end{equation}
Define $f_m=G_{k}^{-1}\left(h_m\right)$.  Then $\{f_m\}$ is an extremizing sequence for inequality \eqref{G1}; precisely,
$$\lim_{m\to\infty}\frac{\|f_m\|_{L^{\frac{2n}{n-2k}}(\mathbb{B}^n)}}{\| G_k^{\frac{1}{2}}(f_m)\|_{L^{2}(\mathbb{B}^n)}}=C^{-1}_{n,k,\frac{2n}{n-2k}}.$$
\end{lemma}

\begin{proof}
From inequality \eqref{G2}, it is obvious  that
$$\lim_{m\to\infty}\frac{\|G_k^{-\frac{1}{2}}(h_m)\|_{L^{2}(\mathbb{B}^n)}}{\|h_m\|_{L^{\frac{2n}{n+2k}}(\mathbb{B}^n)}}\leq C^{-1}_{n,k,\frac{2n}{n-2k}}.$$
If $\{f_m\}$ satifies condition \eqref{dualsequece1} and we set $G_{k}^{\frac{1}{2}}\left(f_m\right)=g_m$, then $\|g_m\|_{L^2(\mathbb{B}^n)}=1$ and $\lim\limits_{m\to\infty}\|G_{k}^{-\frac{1}{2}}\left(g_m\right)\|_{L^{\frac{2n}{n-2k}}(\mathbb{B}^n)}=C^{-1}_{n,k,\frac{2n}{n-2k}}$. Let $h_m=\left(G_{k}^{-\frac{1}{2}}\left(g_m\right)\right)^{\frac{2n}{n-2k}-1}=\left(f_m\right)^{\frac{n+2k}{n-2k}}$, then
\begin{equation}
\begin{split}
\lim_{m\to\infty}\frac{\|G_k^{-\frac{1}{2}}(h_m)\|_{L^{2}(\mathbb{B}^n)}}{\|h_m\|_{L^{\frac{2n}{n+2k}}(\mathbb{B}^n)}}&\geq \lim_{m\to\infty}\frac{\sup\limits_{\|w\|_{L^2}=1}\int_{\mathbb{B}^n}G_{k}^{-\frac{1}{2}}\left(h_m\right)wdV}{\|h_m\|_{L^{\frac{2n}{n+2k}}(\mathbb{B}^n)}}\\
&\geq \lim_{m\to\infty}\frac{\int_{\mathbb{B}^n}G_{k}^{-\frac{1}{2}}\left(h_m\right)g_mdV}{\|h_m\|_{L^{\frac{2n}{n+2k}}(\mathbb{B}^n)}}\\
&=\lim_{m\to\infty}\frac{\int_{\mathbb{B}^n}h_m G_{k}^{-\frac{1}{2}}\left(g_m\right)dV}{\|h_m\|_{L^{\frac{2n}{n+2k}}(\mathbb{B}^n)}}\\
&=\lim_{m\to\infty}\frac{\|h_m\|_{L^{\frac{2n}{n+2k}}(\mathbb{B}^n)}\|G_{k}^{-\frac{1}{2}}\left(g_m\right)\|_{L^{\frac{2n}{n-2k}}(\mathbb{B}^n)}}{\|h_m\|_{L^{\frac{2n}{n+2k}}(\mathbb{B}^n)}}\\
&=C^{-1}_{n,k,\frac{2n}{n-2k}}.
\end{split}
\end{equation}
On the other hand, if $\{h_m\}$ satisfies condition \eqref{dualsequece2} and we set $g_m=G_{k}^{-\frac{1}{2}}\left(h_m\right)$, then
\begin{equation}
\begin{split}
\lim_{m\to\infty}\frac{\|G_k^{-\frac{1}{2}}(g_m)\|_{L^{\frac{2n}{n-2k}}(\mathbb{B}^n)}}{\|g_m\|_{L^{2}(\mathbb{B}^n)}}&\geq\lim_{m\to\infty}\frac{\sup\limits_{\|w\|_{L^{\frac{2n}{n+2k}}}=1}\int_{\mathbb{B}^n}G_{k}^{-\frac{1}{2}}\left(g_m\right)wdV}{\|g_m\|_{L^{2}(\mathbb{B}^n)}}\\
&\geq\lim_{m\to\infty}\frac{\int_{\mathbb{B}^n}G_{k}^{-\frac{1}{2}}\left(g_m\right)h_mdV}{\|g_m\|_{L^{2}(\mathbb{B}^n)}}\\
&=\lim_{m\to\infty}\frac{\int_{\mathbb{B}^n}g_m G_{k}^{-\frac{1}{2}}\left(h_m\right)dV}{\|g_m\|_{L^{2}(\mathbb{B}^n)}}\\
&=\lim_{m\to\infty} \|g_m\|_{L^{2}(\mathbb{B}^n)}\\&=\lim_{m\to\infty}\|G_k^{-\frac{1}{2}}(h_m)\|_{L^{2}(\mathbb{B}^n)}\\
&=C^{-1}_{n,k,\frac{2n}{n-2k}}.
\end{split}
\end{equation}
So, $\{g_m\}$ is the extremizing sequence for inequality \eqref{reducepoincare}. Let $f_m=G_{k}^{-\frac{1}{2}}\left(g_m\right)=G_{k}^{-1}\left(h_m\right)$, then $\{f_m\}$ is the extremizing sequence for inequality \eqref{G1}.
\end{proof}

\begin{lemma}\label{kernel-decreasing}
 Let  $G^{-1}_{k}\left(x,y\right)$ be the kernel function of the operator $G^{-1}_k$ on the hyperbolic ball $\mathbb{B}^n$. Then, $G^{-1}_{k}\left(x,y\right)$ is a positive radially strictly decreasing function with respect to the geodesic distance  $\rho\left(T_{x}\left(y\right)\right)$.
 \end{lemma}

 \begin{proof}
 According to the polynomial decomposition theorem, we can write
$$G_{k}\left(f\right)=\left(P_k-\prod_{i=1}^{k}\frac{\left(2i-1\right)^2}{4}\right)\left(f\right)=\left(P_1-\lambda_1\right)\left(P_1-\lambda_2\right)\cdot\cdot\cdot\left(P_1-\lambda_k\right)\left(f\right),$$
where $\lambda_1, \cdot\cdot\cdot \lambda_k$ are roots of the polynomial equation
$$x\left(x+2\right)\cdot\cdot\cdot\left(x+k\left(k-1\right)\right)-\prod_{i=1}^{k}\frac{\left(2i-1\right)^2}{4}=0.$$
 If $\lambda_j$ is a real number, then $\lambda_j\leq \frac{1}{4}$, otherwise $\lambda_j\left(\lambda_j+2\right)\cdot\cdot\cdot\left(\lambda_j+k\left(k-1\right)\right)>\prod_{i=1}^{k}\frac{\left(2i-1\right)^2}{4}$. For those complex $\lambda_j$, we can obtain $Re\lambda_j\leq \frac{1}{4}$, otherwise
 \begin{equation}\begin{split}
 \prod_{i=1}^{k}\frac{\left(2i-1\right)^2}{4}=&\lambda_j\left(\lambda_j+2\right)\cdot\cdot\cdot\left(\lambda_j+k\left(k-1\right)\right)\\
 =&|\lambda_j\left(\lambda_j+2\right)\cdot\cdot\cdot\left(\lambda_j+k\left(k-1\right)\right)|\\
 =&|\lambda_j||\left(\lambda_j+2\right)|\cdot\cdot\cdot|\left(\lambda_j+k\left(k-1\right)\right)|\\
 \geq&Re\lambda_j \left(Re\lambda_j+2\right)\cdot\cdot\cdot  \left(Re\lambda_j+k\left(k-1\right)\right)\\
 > &\prod_{i=1}^{k}\frac{\left(2i-1\right)^2}{4},
\end{split} \end{equation}
thus it leads to a contradiction.
Let $-\frac{n^
2-2n}{4}-\lambda_j=\tilde{\lambda_j}$, then $\left(P_1-\lambda_j\right)^{-1}=\left(\tilde{\lambda_j}-\Delta_{\mathbb{H}^n}\right)^{-1}$. According to Lu and Yang's work in \cite{LY19, LY22}, we know that the explicit formula of kernel function $\left(\tilde{\lambda_j}-\Delta_{\mathbb{H}^n}\right)^{-1}$ for real valued $\lambda_j\leq \frac{1}{4}$ is as follows:
$$\left(\tilde{\lambda_j}-\Delta_{\mathbb{H}^n}\right)^{-1}=\frac{A_n}{\left(\sinh\rho\right)^{n-2}}\int_{0}^{\pi}\left(\cosh\rho+\cos t\right)^{\frac{n-4}{2}-\theta_{n}\left(\tilde{\lambda_j}\right)}\left(\sin t\right)^{2\theta_n\left(\tilde{\lambda_j}\right)+1}dt,\ \ n\geq 3$$
where
$$\theta_{n}\left(\tilde{\lambda_j}\right)=\sqrt{\tilde{\lambda_j}+\frac{\left(n-1\right)^2}{4}}-\frac{1}{2},$$
$$A_n=\left(2\pi\right)^{-\frac{n}{2}}\frac{\Gamma\left(\frac{
n}{2}+\theta_{n}\left(\tilde{\lambda_j}\right)\right)}{2^{\theta_{n}\left(\tilde{\lambda_j}\right)+1}\Gamma\left(\theta_{n}\left(\tilde{\lambda_j}\right)+1\right)}.$$
Then
\begin{equation}\begin{split}
&\frac{d}{d\rho}\left(\tilde{\lambda_j}-\Delta_{\mathbb{H}^n}\right)^{-1}\\
&\ \ =\frac{A_n}{\sinh^{n-2}\rho}\int_0^\pi\left(\left(2-n\right)\frac{\cosh\rho}{\sinh\rho}+\left(\frac{n-4}{2}-\theta_{n}\left(\tilde{\lambda_j}\right)\right)
\frac{\sinh \rho}{\cosh\rho+\cos t}\right)\cdot\\
&\ \ \ \ \left(\cosh\rho+\cos t\right)^{\frac{n-4}{2}-\theta_{n}\left(\tilde{\lambda_j}\right)}\left(\sin t\right)^{2\theta_{n}\left(\tilde{\lambda_j}\right)+1}dt.
\end{split}\end{equation}
Let
\begin{equation}\begin{split}
f\left(\rho\right)=&\left(2-n\right)\frac{\cosh\rho}{\sinh\rho}+\left(\frac{n-4}{2}-\theta_{n}\left(\tilde{\lambda_j}\right)\right)
\frac{\sinh \rho}{\cosh\rho+\cos t}\\
\leq& \left(2-n\right)\frac{\cosh\rho}{\sinh\rho}+\left(\frac{n-3}{2}\right)
\frac{\sinh \rho}{\cosh\rho-1}\\
=&\left(2-n\right)\frac{2\cosh^2\frac{\rho}{2}-1}{2\cosh\frac{\rho}{2}\sinh\frac{\rho}{2}}+\frac{n-3}{2}\frac{2\cosh\frac{\rho}{2}\sinh\frac{\rho}{2}}{2\sinh^2\frac{\rho}{2}}\\
=&\frac{1-n}{2}\frac{\cosh\frac{\rho}{2}}{\sinh\frac{\rho}{2}}+\left(n-2\right)\frac{1}{2\cosh\frac{\rho}{2}\sinh\frac{\rho}{2}}\\
=&\frac{1}{\sinh\frac{\rho}{2}}\left(\frac{1-n}{2}\cosh\frac{\rho}{2}+\left(n-2\right)\frac{1}{2\cosh\frac{\rho}{2}}\right)\\
\leq&\frac{1}{\sinh\frac{\rho}{2}}\left(\frac{1-n}{2}+\frac{n-2}{2}\right)\\
<&0.
\end{split}\end{equation}
Thus, $\left(P_1-\lambda_j\right)^{-1}\left(x,y\right)$ is strictly
decreasing about $\rho\left(T_{x}\left(y\right)\right)$ for $\tilde{\lambda_j}$ for real valued $\lambda_j\leq \frac{1}{4}$.
From the Lemma 6.2 of \cite{LLY}, for the complex valued $\lambda_j$'s case, we know that
$\left(P_1-\lambda_j\right)^{-1}\ast\left(P_1-\overline{\lambda_j}\right)^{-1}$ is positive radially decreasing function,  where $\overline{\lambda_j}$ is the conjugates of $\lambda_j$. In \cite{LLY}, the authors also proved that for any positive functions $H_1\left(x,y\right), H_2\left(x,y\right): \left(\mathbb{B}^n\times \mathbb{B}^n \right)\to \mathbb{R}$, if both of them are decreasing with respect to $\rho\left(x,y\right)$, then $L\left(x,y\right)=\int_{\mathbb{B}^n}H_1\left(x,z\right)H_2\left(z,y\right)dV_Z$ is also decreasing with respect to $\rho\left(x,y\right)$.
Since $$G_{k}^{-1}\left(x,y\right)=\left(P_1-\lambda_k\right)^{-1}\ast \left(P_1-\lambda_{k-1}\right)^{-1}\cdot\cdot\cdot \ast \left(P_1-\lambda_1\right)^{-1},$$
 the kernel function of $G_{k}^{-1}$ is positive and strictly decreasing about $\rho\left(x,y\right)$.
\end{proof}
\medskip

\subsection{Tightness of the minimizing sequence}\label{subsectight}

\begin{proposition}\label{protight}
If $\{f_m\}$ is the radially decreasing minimizing sequence for minimizing  problem \eqref{minimizing}, then $\{f_m\}$ is tight, that is, for any $\epsilon>0$, there exists a $0<R_\epsilon<1$ such that
$$\int_{ B^n\left(0,R_{\epsilon}\right)}|f_m|^{\frac{2n}{n-2k}}dV\geq 1-\epsilon \ \ {\rm for}\ {\rm all}\ m\in\mathbb{N}.$$
\end{proposition}
Before proving the Proposition \ref{protight},  we need the following  first Lions' concentration-compactness lemma on the hyperbolic space for radially decreasing sequence. Following the idea of the proof of Lemma 1.1 in \cite{Lions841}, we can obtain  Lemma \ref{cc1}. We only clarify this Lemma and omit the proof here.

\begin{lemma}\label{cc1}
Let $\{|h_m|^{\frac{2n}{n+2k}}\}$ be a nonnegative  and radially decreasing sequence satisfying:
$$\int_{\mathbb{B}^n}|h_m|^{\frac{2n}{n+2k}}dV=1.$$

Then there exists a subsequence still denoted by $\{|h_m|^{\frac{2n}{n+2k}}\}$ such that one of the following conditions holds:

(a)\ (Compactness)\  For any $\epsilon>0$, there exists $R_{\epsilon}>0$
such that $$\int_{B^n\left(0,R_{\epsilon}\right)}|h_m|^{\frac{2n}{n+2k}}dV\geq 1-\epsilon \ \ {\rm for}\ {\rm all}\ m;$$

(b)\ (Vanishing)\  For all $R>0$, there holds:
$$\lim_{m\rightarrow\infty}\left(\int_{B^n\left(0,R\right)}|h_m|^{\frac{2n}{n+2k}}dV\right)=0;$$

(c)\ (Dichotomy)\ There exists a $\lambda \in\left(0,1\right)$ such that for all $\epsilon >0$, there exist $R$ close to 1, a sequence $\{R_m\}$ satisfying $R_m\to 1$ as $m\to\infty$ such that
$$h^1_m=h_m \chi_{B^n\left(0,R\right)},\ \ h^2_m  = h_m \chi_{\mathbb{B}^n\setminus B^n\left(0,R_{m}\right)},$$
$$\limsup_{m\rightarrow\infty}\left(|\lambda- \int_{\mathbb{B}^n}|h^1_m|^{\frac{2n}{n+2k}}dV|+|\left(1-\lambda\right)-\int_{\mathbb{B}^n}|h_m^2|^{\frac{2n}{n+2k}}dV|\right)\leq \epsilon.$$
\end{lemma}

In order to exclude the vanishing phenomenon for $|h_m|^{\frac{2n}{n+2k}}$, we prove a radial Lemma.

\begin{lemma}\label{radial}
Let $\{f_m\}$ be a radially decreasing minimizing sequence for the critical Poincar\'e-Sobolev inequality on the hyperbolic space. Then there exists a constant $D_{n,k,p}>0$ such that $f_m\left(R\right)\leq D_{n,k,p}\left(1-R\right)^{\frac{n-1}{p}}$ for any $\frac{1}{2}\leq R<1$ and $2<p<\frac{2n}{n-2k}$.
\end{lemma}

\begin{proof}
Recall that $\{f_m\}$ satisfies
$$\lim\limits_{m\rightarrow +\infty}\left(\int_{\mathbb{B}^n}P_k\left(f_m\right)f_mdV-\prod_{i=1}^{k}\frac{\left(2i-1\right)^2}{4}\int_{\mathbb{B}^n}|f_m|^2dV\right)=C_{n,k,\frac{2n}{n-2k}}$$
and
$$\int_{\mathbb{B}^n}|f_m|^{\frac{2n}{n-2k}}dV=1.$$
By high order subcritical Poincar\'e-Sobolev inequality, that is,
$$\left(\int_{\mathbb{B}^n}P_k\left(f_m\right)f_mdV-\prod_{i=1}^{k}\frac{\left(2i-1\right)^2}{4}\int_{\mathbb{B}^n}|f_m|^2dV\right)\geq C_{n,k,p}\left(\int_{\mathbb{B}^n}|f_m|^pdV\right)^{\frac{2}{p}}$$ for any
$2<p<\frac{2n}{n-2k}$,  we obtain that $f_m$ is bounded in $L^p\left(\mathbb{B}^n\right)$ for
$2<p<\frac{2n}{n-2k}$. By the radial monotonicity of $f_m$, for $0<R<1$, we have
$$f^p_m\left(R\right)\int_{B^n\left(0,R\right)}\left(\frac{2}{1-|x|^2}\right)^ndx\leq \int_{B^n\left(0,1\right)}f_m^p\left(x\right) \left(\frac{2}{1-|x|^2}\right)^ndx\lesssim 1.$$
Then we have  $f_m^p\left(R\right)w_{n-1}\int_{0}^{R}\left(\frac{2}{1-s^2}\right)^n s^{n-1}ds\lesssim 1$. According to L'Hospital's Rule, it is not difficult to check that $\int_{0}^{R}\left(\frac{2}{1-s^2}\right)^n s^{n-1}ds$ is equivalent to $\frac{1}{\left(1-R\right)^{n-1}}$ when $R$ approaches to $1$.
Hence there exists $D_{n,k,p}>0$ such that $f_m^p\left(R\right)\left(\frac{1}{1-R}\right)^{n-1} \leq D_{n,k,p}^p$ when $\frac{1}{2}\leq R<1 $, that is, $f_m\left(R\right)\leq D_{n,k,p}\left(1-R\right)^{\frac{n-1}{p}}$ for $\frac{1}{2}\leq R<1 $.
\end{proof}

Now we turn to prove Proposition \ref{protight}.

{\it Proof of Proposition \ref{protight}.}
$\{f_m\}$ is the radially decreasing minimizing sequence for the critical Poincar\'e-Sobolev inequality on the hyperbolic space.
According to Lemma \ref{relationshipsequence}, we know that $h_m=(f_m)^{\frac{n+2k}{n-2k}}$ is the maximinzing sequence for the Hardy-Littlewood-Sobolev inequality \eqref{BHLS} on the hyperbolic space. We need to note that
$$|h_m|^{\frac{2n}{n+2k}}=|f_m|^{\frac{2n}{n-2k}}.$$
Once we proved the tightness for the maximizing sequence $|h_m|^{\frac{2n}{n+2k}}$, we can derive the tightness for the minimizing sequence $|f_m|^{\frac{2n}{n-2k}}$ of the critical Poincar\'e-Sobolev inequality on the hyperbolic space. We obtain this through excluding the the vanishing phenomenon and dichotomy phenomenon for maximizing sequence $|h_m|^{\frac{2n}{n+2k}}$ of Hardy-Littlewood-Sobolev inequality \eqref{BHLS} on the hyperbolic space.

We first  exclude the vanishing phenomenon.
Through  Lemma \ref{relationshipsequence} and \ref{radial}, we have that
$$h_m\left(R\right)\leq \left(D_{n,k,p} \left(1-R\right)^{\frac{n-1}{p}}\right)^{\frac{n+2k}{n-2k}} \ \ \text{for} \ \ \frac{1}{2}\leq R<1,\ \  2<p<\frac{2n}{n-2k}.$$
 Suppose that vanishing phenomenon occurs for  $\{h_m\}$, that is,
$$\lim\limits_{m\rightarrow +\infty}\int_{B^n\left(0,R\right)}|h_m|^{\frac{2n}{n+2k}}dV=0\ \  \text{for all} \ 0< R<1, $$
which follows that 
\begin{equation}\label{outerest}
\lim\limits_{R\rightarrow 1}\lim\limits_{m\rightarrow +\infty}\int_{B^n\left(0,R\right)}|h_m|^{\frac{2n}{n+2k}}\left(\frac{2}{1-|x|^2}\right)^ndx=0.   
\end{equation}
It is not difficult to check that for $2<p<\frac{2n}{n-2k}$,
$$\lim\limits_{R\rightarrow 1} \int_{B^n\left(0,1\right) \setminus B^n\left(0,R\right)}|\left(1-|x|\right)^{\frac{n-1}{p}}|^{\frac{2n}{n-2k}}\left(\frac{2}{1-|x|^2}\right)^ndx=0,$$
which implies that $$\lim\limits_{R\rightarrow 1}\lim\limits_{m\rightarrow +\infty}\int_{B^n\left(0,1\right)\setminus B^n\left(0,R\right)}|h_m|^{\frac{2n}{n+2k}}\left(\frac{2}{1-|x|^2}\right)^ndx=0.$$
Combining this and \eqref{outerest}, we conclude that $$\lim\limits_{m\rightarrow +\infty}\int_{\mathbb{B}^n}|h_m|^{\frac{2n}{n+2k}}dV=0,$$ which contradicts $\int_{\mathbb{B}^n}|h_m|^{\frac{2n}{n+2k}}dV=1$.

\medskip
 We continue to eliminate dichotomy phenomenon. If case (c) occurs, we derive

$$\lim_{m\rightarrow\infty} \int_{\mathbb{B}^n}|h^1_m|^{\frac{2n}{n+2k}}dV=\lambda,\ \ \lim_{m\rightarrow\infty}\int_{\mathbb{B}^n}|h_m^2|^{\frac{2n}{n+2k}}dV=1-\lambda.$$

Let $h^3_m =f_m \chi_{B^n\left(0,R_{m}\right)\setminus B^n\left(0,R\right)}$, then
$$\lim_{m\rightarrow\infty} \int_{\mathbb{B}^n}|h^3_m|^{\frac{2n}{n+2k}}dV=0.$$ We now claim that
\small{\begin{equation}\begin{split}\label{dic}
&\lim\limits_{m\rightarrow +\infty}\int_{\mathbb{B}^n}\int_{\mathbb{B}^n}h_m\left(x\right)G_k^{-1}\left(x,y\right)h_m\left(y\right)dV_{x}dV_{y}\\
&=\lim\limits_{m\rightarrow +\infty}\int_{\mathbb{B}^n}\int_{\mathbb{B}^n}h_m^1\left(x\right)G_k^{-1}\left(x,y\right)h_m^1\left(y\right)dV_{x}dV_{y}+
\lim\limits_{m\rightarrow +\infty}\int_{\mathbb{B}^n}\int_{\mathbb{B}^n}h_m^2\left(x\right)G_k^{-1}\left(x,y\right)h_m^2\left(y\right)dV_{x}dV_{y}.\\
\end{split}\end{equation}}
In fact, we first have
\begin{equation}\begin{split}
&\lim\limits_{m\rightarrow +\infty}\int_{\mathbb{B}^n}\int_{\mathbb{B}^n}h_m\left(x\right)G_k^{-1}\left(x,y\right)h_m\left(y\right)dV_{x}dV_{y}\\
&=\lim\limits_{m\rightarrow +\infty}\int_{\mathbb{B}^n}\int_{\mathbb{B}^n}h_m^1\left(x\right)G_k^{-1}\left(x,y\right)h_m^1\left(y\right)dV_{x}dV_{y}+
\lim\limits_{m\rightarrow +\infty}\int_{\mathbb{B}^n}\int_{\mathbb{B}^n}h_m^2\left(x\right)G_k^{-1}\left(x,y\right)h_m^2\left(y\right)dV_{x}dV_{y}\\
&\ \ +2\lim\limits_{m\rightarrow +\infty}\int_{\mathbb{B}^n}\int_{\mathbb{B}^n}h_m^1\left(x\right)G_k^{-1}\left(x,y\right)h_m^2\left(y\right)dV_{x}dV_{y}
\\
&\ \ +2\lim\limits_{m\rightarrow +\infty}\int_{\mathbb{B}^n}\int_{\mathbb{B}^n}\left(h_m^1(x)+h_m^2(x)\right)G_k^{-1}\left(x,y\right)h_m^3\left(y\right)dV_{x}dV_{y}\\
&\ \ + \lim\limits_{m\rightarrow +\infty}\int_{\mathbb{B}^n}\int_{\mathbb{B}^n}h_m^3\left(x\right)G_k^{-1}\left(x,y\right)h_m^3\left(y\right)dV_{x}dV_{y}.
\end{split}\end{equation}
Since
\begin{equation*}\begin{split}
&\lim\limits_{m\rightarrow +\infty}\int_{\mathbb{B}^n}\int_{\mathbb{B}^n}\left(h_m^1(x)+h_m^2(x)\right)G_k^{-1}\left(x,y\right)h_m^3\left(y\right)dV_{x}dV_{y}\\
&\ \ \lesssim \lim\limits_{m\rightarrow +\infty}\left(\int_{\mathbb{B}^n}|h^3_m|^{\frac{2n}{n+2k}}dV\right)^{\frac{n+2k}{2n}}\left(\int_{\mathbb{B}^n}|h^1_m+h_m^2|^{\frac{2n}{n+2k}}dV\right)^{\frac{n+2k}{2n}}=0
\end{split}\end{equation*}
and
$$\lim\limits_{m\rightarrow +\infty}\int_{\mathbb{B}^n}\int_{\mathbb{B}^n}h_m^3\left(x\right)G_k^{-1}\left(x,y\right)h_m^3\left(y\right)dV_{x}dV_{y}
\lesssim \lim\limits_{m\rightarrow +\infty}\left(\int_{\mathbb{B}^n}|h^3_m|^{\frac{2n}{n+2k}}dV\right)^{\frac{n+2k}{n}}=0,$$
in order to prove equality \eqref{dic}, we only need to prove that
$$\lim\limits_{m\rightarrow +\infty}\int_{\mathbb{B}^n}\int_{\mathbb{B}^n}h_m^1\left(x\right)G_k^{-1}\left(x,y\right)h_m^2(y)dV_{x}dV_{y}=0.$$

From Lu-Yang's work in \cite{LY19}, we know that
$$G_k^{-1}\left(x,y\right))\lesssim \left(\frac{1}{\sinh(\rho(x,y))}\right)^{n-2k}.$$
Then it follows that
\begin{equation}\begin{split}
&\lim\limits_{m\rightarrow +\infty}\int_{\mathbb{B}^n}\int_{\mathbb{B}^n}h_m^1\left(x\right)G_k^{-1}\left(x,y\right)h_m^2(y)dV_{x}dV_{y}\\
&\ \ \lesssim \lim\limits_{m\rightarrow +\infty}\int_{\mathbb{B}^n}\int_{\mathbb{B}^n}h_m^1(x)\left(\frac{1}{\sinh(\rho(x,y))}\right)^{n-2k} h_m^2\left(y\right)dV_{x}dV_{y}\\
&\ \ =\lim\limits_{m\rightarrow +\infty}\int_{\mathbb{B}^n}\int_{\mathbb{B}^n}h_m^1(x)\left(\frac{2}{1-|x|^2}\right)^{\frac{n+2k}{2}}|x-y|^{-(n-2k)} h_m^2(y)\left(\frac{2}{1-|y|^2}\right)^{\frac{n+2k}{2}}dxdy\\
&\ \ = \lim\limits_{m\rightarrow +\infty}\int_{\mathbb{B}^n}\int_{\mathbb{B}^n}\widetilde{h_m^1}(x)|x-y|^{-(n-2k)} \widetilde{h_m^2}(y)dxdy
\end{split}\end{equation}
where
$$h_m^1(x)\left(\frac{2}{1-|x|^2}\right)^{\frac{n+2k}{2}}=\widetilde{h_m^1}(x), \ \ h_m^2(y)\left(\frac{2}{1-|y|^2}\right)^{\frac{n+2k}{2}}=\widetilde{h_m^2}(y).$$
When $x\in B^n(0, R)$ and $y\in B^n\setminus B^n(0, R_m)$, there exists $\delta>0$ such that
$$|x-y|^{-(n-2k)} \lesssim |x-y|^{-(n-2k-\delta)}.$$
Then it follows that
\begin{equation*}\begin{split}
&\lim\limits_{m\rightarrow +\infty}\int_{\mathbb{B}^n}\int_{\mathbb{B}^n}\widetilde{h_m^1}(x)|x-y|^{-(n-2k)} \widetilde{h_m^2}(y)dxdy\\
&\ \ \lesssim \lim\limits_{m\rightarrow +\infty}\int_{\mathbb{B}^n}\int_{\mathbb{B}^n}\widetilde{h_m^1}(x)|x-y|^{-(n-2k-\delta)} \widetilde{h_m^2}(y)dxdy\\
&\ \ \lesssim \left(\int_{\mathbb{B}^n}|\widetilde{h_m^1}|^{\frac{2n}{n+2k}}dx\right)^{\frac{n+2k}{2n}} \left(\int_{\mathbb{B}^n}|\widetilde{h_m^2}|^{p_\delta}dx\right)^{\frac{1}{p_\delta}},
\end{split}\end{equation*}
where in the last inequality, we used the Hardy-Littlewood-Sobolev in $\mathbb{R}^n$ with
$$\frac{n+2k}{2n}+\frac{1}{p_\delta}+\frac{n-2k-\delta}{n}=2$$ and $p_\delta<\frac{2n}{n+2k}$.
Careful computation yields
$$\int_{\mathbb{B}^n}|\widetilde{h_m^1}|^{\frac{2n}{n+2k}}dx=\int_{\mathbb{B}^n}|h_m^1|^{\frac{2n}{n+2k}}\left(\frac{2}{1-|x|^2}\right)^ndx\leq \int_{\mathbb{B}^n}|h_m|^{\frac{2n}{n+2k}}dV\leq 1.$$
and
\begin{equation}\begin{split}
\lim\limits_{m\rightarrow +\infty}\int_{\mathbb{B}^n}|\widetilde{h_m^2}|^{p_\delta}dx&=\lim\limits_{m\rightarrow +\infty}\int_{\mathbb{B}^n\setminus B^n(0,R_m)}|h_m^2|^{p_\delta}\left(\frac{2}{1-|x|^2}\right)^{\frac{n+2k}{2}p_\delta}dV\\
&=\lim\limits_{m\rightarrow +\infty}\int_{\mathbb{B}^n\setminus B^n(0,R_m)}|h_m^2|^{p_\delta}\left(\frac{2}{1-|x|^2}\right)^{n} \left(\frac{2}{1-|x|^2}\right)^{\frac{n+2k}{2}p_\delta-n}dx\\
&\ \ \leq \lim\limits_{m\rightarrow +\infty}\int_{\mathbb{B}^n\setminus B^n(0,R_m)}|h_m^2|^{p_\delta}dV \left(\frac{2}{1-|R_m|^2}\right)^{\frac{n+2k}{2}p_\delta-n}=0,
\end{split}\end{equation}
where the last equality holds since $$\lim\limits_{m\rightarrow +\infty} \left(\frac{2}{1-|R_m|^2}\right)^{\frac{n+2k}{2}p_\delta-n}=0$$ and
$$\int_{\mathbb{B}^n\setminus B^n(0,R_m)}|h_m^2|^{p_\delta}dV\leq \int_{\mathbb{B}^n\setminus B^n(0,R_m)}|f_m|^{\frac{n+2k}{n-2k}p_\delta}dV\lesssim \left(\int_{\mathbb{B}^n}P_k\left(f_m\right)f_mdV-\prod_{i=1}^{k}\frac{\left(2i-1\right)^2}{4}\int_{\mathbb{B}^n}|f_m|^2dV\right)^{\frac{1}{2}} $$ through inequality \eqref{BPS}. Combining the above estimate, we prove that
\small{\begin{equation*}\begin{split}
&\lim\limits_{m\rightarrow +\infty}\int_{\mathbb{B}^n}\int_{\mathbb{B}^n}h_m\left(x\right)G_k^{-1}\left(x,y\right)h_m\left(y\right)dV_{x}dV_{y}\\
&=\lim\limits_{m\rightarrow +\infty}\int_{\mathbb{B}^n}\int_{\mathbb{B}^n}h_m^1\left(x\right)G_k^{-1}\left(x,y\right)h_m^1\left(y\right)dV_{x}dV_{y}+
\lim\limits_{m\rightarrow +\infty}\int_{\mathbb{B}^n}\int_{\mathbb{B}^n}h_m^2\left(x\right)G_k^{-1}\left(x,y\right)h_m^2\left(y\right)dV_{x}dV_{y}.\\
\end{split}\end{equation*}}

Then it follows from the Hardy-Littlewood-Sobolev inequality \eqref{BHLS} in hyperbolic space  that
\begin{equation}\begin{split}
&\lim\limits_{m\rightarrow +\infty}\int_{\mathbb{B}^n}\int_{\mathbb{B}^n}h_m\left(x\right)G_k^{-1}\left(x,y\right)h_m\left(y\right)dV_{x}dV_{y}\\
&\ \ \leq C_{n,k,\frac{2n}{n-2k}}^{-1}\lim\limits_{m\rightarrow +\infty}\left(\int_{\mathbb{B}^n}|h_m^1|^{\frac{2n}{n+2k}}\right)^{\frac{n+2k}{n}}+C_{n,k,\frac{2n}{n-2k}}^{-1}\lim\limits_{m\rightarrow +\infty}\left(\int_{\mathbb{B}^n}|h_m^1|^{\frac{2n}{n+2k}}\right)^{\frac{n+2k}{n}}\\
&\ \ =C_{n,k,\frac{2n}{n-2k}}^{-1}\lambda^{\frac{n+2k}{n}}+C_{n,k,\frac{2n}{n-2k}}^{-1}\left(1-\lambda\right)^{\frac{n+2k}{n}}\\
&\ \ <C_{n,k,\frac{2n}{n-2k}}^{-1},
\end{split}\end{equation}
which is a contradiction with $$\lim\limits_{m\rightarrow +\infty}\int_{\mathbb{B}^n}\int_{\mathbb{B}^n}h_m\left(x\right)G_k^{-1}\left(x,y\right)h_m\left(y\right)dV_{x}dV_{y}=C_{n,k,\frac{2n}{n-2k}}^{-1}.$$
Thus we complete the proof of the Proposition \ref{protight}.
$\hfill\square$
\medskip

\subsection{$\{f_m\}$ converges strongly to $f_0\neq 0$}\label{subconvergence}

\begin{proposition}\label{proconvergence}
If $\{f_m\}$ is tight and  radially decreasing minimizing sequence for minimizing  problem \eqref{minimizing}, then $\{f_m\}$ converges strongly to a function $f_0\neq 0$ in $H(\mathbb{B}^n)$.
\end{proposition}
 Before proving the Proposition \ref{proconvergence}, we establish the 
second Lions' concentration compactness Lemma on the hyperbolic space $\mathbb{B}^n$ for the high order critical Poincar\'e-Sobolev inequality.

Since $\|f_m\|_{H\left(\mathbb{B}^n\right)}$ is uniformly bounded, we can assume that $f_m\to f_0$ weakly in $H\left(\mathbb{B}^n\right)$. Pick $\phi \in C_{c}^{\infty}\left(\mathbb{B}^n\right)$ satisfying $\phi|_{B^n\left(0, 2R\right)}=1$ and $0\leq\phi\leq1$,
applying the Poincar\'e-Sobolev inequality on hyperbolic space $\mathbb{B}^n$ into $f_m\phi$, we see that
$$\left(\int_{\mathbb{B}^n}P_k\left(f_m\phi\right)\left(f_m\phi\right)dV-\prod_{i=1}^{k}\frac{\left(2i-1\right)^2}{4}\int_{\mathbb{B}^n}|f_m \phi|^2dV\right)\geq C_{n,k}\left(\int_{\mathbb{B}^n}|f_m\phi|^{\frac{2n}{n-2k}}\right)^{\frac{n-2k}{n}}.$$
Since $\phi$ has the compact support, hence we can deduce that there exists $C_{n,k, supp\left(\phi\right)}>0$ such that
$$\left(\int_{\mathbb{B}^n}P_k\left(f_m\phi\right)\left(f_m\phi\right)dV-\prod_{i=1}^{k}\frac{\left(2i-1\right)^2}{4}\int_{\mathbb{B}^n}|f_m \phi|^2dV\right)\geq C_{n,k, supp \left(\phi\right)}\int_{\mathbb{B}^n}|f_m\phi|^2dV.$$
Then we have that $$\left(\int_{\mathbb{B}^n}P_k\left(f_m\phi\right)\left(f_m\phi\right)dV-\prod_{i=1}^{k}\frac{\left(2i-1\right)^2}{4}\int_{\mathbb{B}^n}|f_m \phi|^2dV\right)$$ is equivalent to $$\int_{\mathbb{B}^n}P_k\left(f_m\phi\right)\left(f_m\phi\right)dV.$$
This deduces that $f_m\phi$ is bounded in high order Sobolev space $W^{k,2}\left(\mathbb{B}^n\right)$ which implies that $f_m$ is bounded in $W^{k,2}_{loc}\left(\mathbb{B}^n\right)$. That is, $$H\left(\mathbb{B}^n\right)\subset W^{k,2}_{loc}\left(\mathbb{B}^n\right).$$

The compact embedding theorem on $\mathbb{B}^n$ deduces  that $f_m\to f_0$ strongly in  $W^{t,2}_{loc}\left(\mathbb{B}^n\right)$ for $0<t<k$ so that $f_m\left(x\right)\to f_0\left(x\right), a.e.\ x\in \mathbb{B}^n$.

\medskip
We  prove the following second concentration compactness principle on the hyperbolic ball $\mathbb{B}^n$ through similar idea in  \cite{Lions1}.
\begin{lemma}\label{con2}
	Let  $$\mu_m:=|f_m|^{\frac{2n}{n-2k}}dV,\ \ \nu_m:=\left( P_k\left(f_m\right)f_m-\prod_{i=1}^{k}\frac{\left(2i-1\right)^2}{4}|f_m|^2\right) dV.$$
	Then $\mu_m,\nu_m$ converge weakly in the sense of measure to some bounded nonnegative measures $\mu$, $\nu$ on $\mathbb{B}^n$. We also assume that $\{f_m\}$ converge weakly to $f_0$  in $H\left(\mathbb{B}^n\right)$  and $\{|f_m|^{\frac{2n}{n-2k}}\}$ is tight. 
     Then there exists  at most countable set of points $\{x_i\}_{i\in I} \subset\mathbb{B}^n$ such that:
	$$\nu\geq\left( P_k\left(f_0\right)f_0-\prod_{i=1}^{k}\frac{\left(2i-1\right)^2}{4}|f_0|^2\right)dV +\sum_{i\in I}v^i \delta_{x_i}$$
	and
	$$\mu= |f_0|^{\frac{2n}{n-2k}}dV+\sum_{i\in I}u^i \delta_{x_i},$$
	where $v^i=\nu\{x_i\}$, $u^i=\mu\{x_i\}$. Furthermore, we also have $v^i\geq C_{n,k}\left(u^i\right)^{\frac{n-2k}{n}}$.
\end{lemma}
Lemma \ref{con2} is proved by applying the following lemma \ref{con2ingre} whose original version is in $\mathbb{R}^n$. Since the crucial ingredient is valid in an arbitrary measure space, we can prove  Lemma \ref{con2ingre} through the same proof process of Lemma 1.2 in \cite{Lions1}. So we omit the proof the Lemma \ref{con2ingre} here.
\begin{lemma}\label{con2ingre}
	Let $\widetilde{\mu},\widetilde{\nu}$ be two bounded nonnegative measures on $\mathbb{B}^n$ satisfying for some constant $C_0\geq 0$
	\begin{equation}
	\left(\int_{\mathbb{B}^n}|\varphi|^q d\widetilde{\mu}\right)^{\frac{1}{q}}\leq C_0 \left(\int_{\mathbb B^n}|\varphi|^p d\widetilde{\nu}\right)^{\frac{1}{p}}, \ \ \text{for any}\  \varphi\in C_c^\infty\left(\mathbb B^n\right)
	\end{equation}
	where $1\leq p< q\leq+\infty$. Then there exist an at most countable set $I$, families $\{x_i\}$ of distinct points in $\mathbb B^n$, $\{v^i\}$ in $[0,+\infty]$ such that
	\begin{equation}
	\widetilde{\mu}=\sum_{i\in I} v^i \delta_{x_i},\ \ \ \widetilde{\nu}\geq C_0^{-p}\sum_{i\in I} \left(v^i\right)^{\frac{p}{q}}\delta_{x_i}.
	\end{equation}
\end{lemma}

{\it Proof of Lemma \ref{con2}.} Through calculating, for  any $ \varphi\in C_c^\infty\left(\mathbb B^n\right) $, we have the following formula{\tiny }

\begin{equation}\begin{split}\label{pkspread}
&\int_{\mathbb{B}^n}P_k\left(f_m \varphi \right)f_m \varphi dV-\prod_{i=1}^{k}\frac{\left(2i-1\right)^2}{4}\int_{\mathbb{B}^n}|f_m \varphi|^2dV\\
&\ \ =\int_{\mathbb{B}^n} \sum_{l=1}^{k}c_l\left(-\Delta_{\mathbb{H}}\right)^l\left( f_m \varphi\right) \left( f_m \varphi\right)dV-\prod_{i=1}^{k}\frac{\left(2i-1\right)^2}{4}\int_{\mathbb{B}^n}|f_m \varphi|^2dV\\
&\ \ =\int_{\mathbb{B}^n}P_k\left(f_m \right)f_m \varphi^2 dV-\prod_{i=1}^{k}\frac{\left(2i-1\right)^2}{4}\int_{\mathbb{B}^n}|f_m \varphi|^2dV+\int_{\mathbb{B}^n} \sum_{l=1}^{k}c_l\sum_{j= 1}^{2l}\nabla_{\mathbb{H}}^{2l-j} f_m \nabla_{\mathbb{H}}^j \varphi \left(f_m \varphi\right) dV.\\	
\end{split}\end{equation}

If we denote by $g_m=f_m-f_0$, then $\|g_m\|_{L^{\frac{2n}{n-2k}}}$ and $\|g_m\|_{H}$ are uniformly bounded. let
$$\widetilde{\mu_m}:=|g_m|^{\frac{2n}{n-2k}}dV,\ \ \widetilde{\nu_m}:=\left( P_k\left(g_m\right)g_m-\prod_{i=1}^{k}\frac{\left(2i-1\right)^2}{4}|g_m|^2\right) dV,$$
then there exist two bounded measure $\widetilde{\mu}$ and $\widetilde{\nu}$ such that $\widetilde{\mu_m},\widetilde{\nu_m}$ converge weakly in the sense of measure to  $\widetilde{\mu}$, $\widetilde{\nu}$  on $\mathbb{B}^n$.  Poincar\'e-Sobolev inequality, formula \eqref{pkspread} and compact embedding theorem deduce that 
\begin{equation}\begin{split}
&\large\left(\int_{\mathbb{B}^n}|g_m\varphi|^{\frac{2n}{n-2k}}dV\large \right)^{\frac{n-2k}{n}} \\
&\ \ \leq C_{n,k}  \int_{\mathbb{B}^n}P_k\left(g_m\varphi \right)\left(g_m\varphi \right) dV-\prod_{i=1}^{k}\frac{\left(2i-1\right)^2}{4}\int_{\mathbb{B}^n}|g_m\varphi|^2dV\\
&\ \ =\int_{\mathbb{B}^n}P_k\left(g_m \right)g_m \varphi^2 dV-\prod_{i=1}^{k}\frac{\left(2i-1\right)^2}{4}\int_{\mathbb{B}^n}|g_m \varphi|^2dV,\\
\end{split}\end{equation}
when $m\to\infty$.

Theorefore, for any $\varphi\in C_c^\infty\left(\mathbb B^n\right)$, we have
\begin{equation*}
\left(\int_{\mathbb{B}^n}|\varphi|^{\frac{2n}{n-2k}} d\widetilde{\nu}\right)^{\frac{n-2k}{2n}}\leq C_0 \left(\int_{\mathbb B^n}|\varphi|^2 d\widetilde{\mu}\right)^{\frac{1}{2}},
\end{equation*}
According to Lemma \ref{con2ingre},
$\widetilde{\nu}=\sum_{i\in I} u^i \delta_{x_i}$.
Br\'{e}zis-Lieb Lemma yields that  for any  $\varphi\in C_c^\infty\left(\mathbb B^n\right)$,
\begin{equation*}
\int_{\mathbb{B}^n}|f_m\varphi|^{\frac{2n}{n-2k}}dV-\int_{\mathbb{B}^n}|g_m\varphi|^{\frac{2n}{n-2k}}dV\to \int_{\mathbb{B}^n}|f_0\varphi|^{\frac{2n}{n-2k}}dV.
\end{equation*}
Therefore, we have $$\mu= |f_0|^{\frac{2n}{n-2k}}+\sum_{i\in I}u^i \delta_{x_i}.$$

Next, for fomula \eqref{pkspread}, by integrating by parts, we have
\begin{equation}\begin{split}\label{intbypart}
&\int_{\mathbb{B}^n} \sum_{l=1}^{k}c_l\sum_{j= 1}^{2l}\nabla_{\mathbb{H}}^{2l-j} f_m \nabla_{\mathbb{H}}^j \varphi \left(f_m \varphi\right) dV\\
&\ \ =\int_{\mathbb{B}^n} \sum_{l=1}^{k}\sum_{\alpha+\beta+\gamma+\delta=2l}c_{l,\alpha,\beta,\gamma,\delta}\nabla_{\mathbb{H}}^\alpha f_m \nabla_{\mathbb{H}}^{\beta} \varphi\nabla_{\mathbb{H}}^\gamma f_m \nabla_{\mathbb{H}}^{\delta} \varphi dV,\\
\end{split}\end{equation}
where $0\leq\gamma\leq\alpha\leq l$, $\alpha+\gamma\leq2l-1$.

Furthermore,
\begin{equation}\begin{split}\label{intbypart2}
&\int_{\mathbb{B}^n} \sum_{l=1}^{k}\sum_{\alpha+\beta+\gamma+\delta=2l}c_{l,\alpha,\beta,\gamma,\delta}\nabla_{\mathbb{H}}^\alpha f_m \nabla_{\mathbb{H}}^{\beta} \varphi\nabla_{\mathbb{H}}^\gamma f_m \nabla_{\mathbb{H}}^{\delta} \varphi \\ &
\ \ -\sum_{l=1}^{k}\sum_{\alpha+\beta+\gamma+\delta=2l}c_{l,\alpha,\beta,\gamma,\delta}\nabla_{\mathbb{H}}^\alpha f \nabla_{\mathbb{H}}^{\beta} \varphi\nabla_{\mathbb{H}}^\gamma f \nabla_{\mathbb{H}}^{\delta} \varphi dV\\
&\ \ \ \ =\int_{\mathbb{B}^n} \sum_{l=1}^{k}\sum_{\alpha+\beta+\gamma+\delta=2l}c_{l,\alpha,\beta,\gamma,\delta}\nabla_{\mathbb{H}}^\alpha f_m \nabla_{\mathbb{H}}^{\beta} \varphi\left(\nabla_{\mathbb{H}}^\gamma f_m-\nabla_{\mathbb{H}}^\gamma f\right) \nabla_{\mathbb{H}}^{\delta} \varphi \\
&\ \ \ \ \ \  +\sum_{l=1}^{k}\sum_{\alpha+\beta+\gamma+\delta=2l}c_{l,\alpha,\beta,\gamma,\delta}\left(\nabla_{\mathbb{H}}^\alpha f_m-\nabla_{\mathbb{H}}^\alpha f\right) \nabla_{\mathbb{H}}^{\beta} \varphi\nabla_{\mathbb{H}}^\gamma f \nabla_{\mathbb{H}}^{\delta} \varphi dV\\
&\ \ \ \ = \uppercase\expandafter{\romannumeral1}+\uppercase\expandafter{\romannumeral2}.
\end{split}\end{equation}
Since $f_m\to f$ strongly in $W^{t,2}_{loc}\left(\mathbb{B}^n\right)$ for $t<k$ and $\|f_m\|_{W^{k,2}_{loc}}\leq\|f_m\|_{H}\leq C $,
\begin{equation*}
\uppercase\expandafter{\romannumeral1}\leq \|\nabla_{\mathbb{H}}^\alpha f_m\|_{L^2_{loc}\left(\mathbb{B}^n\right)}  \|\nabla_{\mathbb{H}}^\alpha f_m-\nabla_{\mathbb{H}}^\alpha f\|_{L^2_{loc}\left(\mathbb{B}^n\right)} \to 0. \\
\end{equation*}
Since $f_m\to f$ weakly in $W^{t,2}_{loc}\left(\mathbb{B}^n\right)$
for $t\leq k$ and $\|f\|_{W^{k,2}_{loc}}\leq\|f\|_{H}\leq C $, then $\uppercase\expandafter{\romannumeral2}\to 0$.
Passing to the limit in \eqref{pkspread}, we find for any $\varphi\in C_c^\infty\left(\mathbb B^n\right)$,
\begin{equation*}\begin{split}
&\large\left(\int_{\mathbb{B}^n}|\varphi|^{\frac{2n}{n-2k}}d\mu \large \right)^{\frac{n-2k}{n}}\cdot \left(C_{n,k}\right)^{\frac{n-2k}{n}}\\
&\ \ \leq\int_{\mathbb B^n}|\varphi|^2 d\nu+\int_{\mathbb{B}^n} \sum_{l=1}^{k}\sum_{\alpha+\beta+\gamma+\delta=2l}c_{l,\alpha,\beta,\gamma,\delta}\nabla_{\mathbb{H}}^\alpha f \nabla_{\mathbb{H}}^{\beta} \varphi\nabla_{\mathbb{H}}^\gamma f \nabla_{\mathbb{H}}^{\delta} \varphi dV.
\end{split}\end{equation*}
Taking $\varphi_{\epsilon}^0=\varphi(\frac{x}{\epsilon})\in C_{c}^{\infty}(B_{\mathbb{H}}\left(0,\epsilon\right))$ satisfying $0\leq \varphi_{\epsilon}^{0}\leq 1$, $\varphi_{\epsilon}^0=1$ in $B_{\mathbb{H}}\left(0,\frac{\epsilon}{2}\right)$, we can easily check that
$|\nabla_{\mathbb{H}}\varphi_{\epsilon}^0|=|\frac{\left(1-|x|^2\right)}{2}\nabla_{\mathbb{R}^n}\varphi_{\epsilon}^0|\lesssim \epsilon^{-1}$. Define $\varphi_{\epsilon}^{x_i}$ by
$\varphi_{\epsilon}^{x_i}=\varphi_{\epsilon}^0(T_{x_i}(x))$, by the isometry of Mobius transformation $T_{x_i}$, we can immediately derive that
$\varphi_{\epsilon}^{x_i}\in C_{c}^{\infty}(B_{\mathbb{H}}\left(x_i,\epsilon\right))$ satisfying $0\leq \varphi_{\epsilon}^{x_i}\leq 1$, $\varphi_{\epsilon}^{x_i}=1$ in $B_{\mathbb{H}}\left(x_i,\frac{\epsilon}{2}\right)$ and $|\nabla_{\mathbb{H}}\varphi_{\epsilon}^{x_i}|\lesssim \epsilon^{-1}$. We apply the above
with $\varphi_{\epsilon}^{x_i}$ for sufficiently small $\epsilon>0$ and $i\in I$, then

\begin{equation*}\begin{split}
&\large\left(\int_{\mathbb{B}^n}|\varphi_{\epsilon}^{x_i}|^{\frac{2n}{n-2k}}d\mu \large \right)^{\frac{n-2k}{n}}\cdot C_{n,k}\\
&\ \ \leq\int_{\mathbb B^n}|\varphi_{\epsilon}^{x_i}|^2 d\nu\\
&\ \ \ \ +\int_{B_{\mathbb{H}\left(x_i,\epsilon\right)}}  \sum_{\alpha+\beta+\gamma+\delta=2l}c_{l,\alpha,\beta,\gamma,\delta}\epsilon^{-\left(\beta+\delta\right)}\nabla_{\mathbb{H}}^\alpha f \left(\nabla_{\mathbb{H}}^{\beta} \varphi\right)\left(\frac{T_{x_i}(x)}{\epsilon}\right)\nabla_{\mathbb{H}}^\gamma f \left(\nabla_{\mathbb{H}}^{\delta} \varphi\right)\left(\frac{T_{x_i}(x)}{\epsilon}\right) dV.
\end{split}\end{equation*}
Recall that we have the following Sobolev inequality on the hyperbolic space $\mathbb{B}^n$:
\begin{equation}\label{sobolev2}
\left(\int_{\mathbb{B}^n}|\nabla_{\mathbb{H}}^j f|^{q_j}dV\right)^{\frac{1}{q_j}}\leq\left(\int_{\mathbb{B}^n}|\nabla_{\mathbb{H}}^l f|^{2}dV\right)^{\frac{1}{2}},
\end{equation}
where $j=0,1,2,\cdots,l-1$, $q_j=\frac{2n}{n-2\left(l-j\right)}$. By using the H\"{o}lder inequality and inequality \eqref{sobolev2}, we have the following estimate:
\begin{equation*}\begin{split}
&\int_{B_{\mathbb{H}}\left(x_i,\epsilon\right)} \sum_{\alpha+\beta+\gamma+\delta=2l}c_{l,\alpha,\beta,\gamma,\delta}\epsilon^{-\left(\beta+\delta\right)}\nabla_{\mathbb{H}}^\alpha f \left(\nabla_{\mathbb{H}}^{\beta} \varphi\right)\left(\frac{T_{x_i}(x)}{\epsilon}\right)\nabla_{\mathbb{H}}^\gamma f \left(\nabla_{\mathbb{H}}^{\delta} \varphi\right)\left(\frac{T_{x_i}(x)}{\epsilon}\right) dV\\
&\ \ \leq  \sum_{\alpha+\beta+\gamma+\delta=2l}c_{l,\alpha,\beta,\gamma,\delta}\epsilon^{-\left(\beta+\delta\right)} \|\nabla_{\mathbb{H}}^\alpha f\|_{L^{\frac{2n}{n-2\left(l-\alpha\right)}}_{B_{\mathbb{H}}\left(x_i,\epsilon\right)}}
\| \left(\nabla_{\mathbb{H}}^{\beta} \varphi\right)\|_{L^{\frac{n}{\beta}}_{B_{\mathbb{H}}\left(x_i,\epsilon\right)}}
\|\nabla_{\mathbb{H}}^\gamma f\|_{L^{\frac{2n}{n-2\left(l-\gamma\right)}}_{B_{\mathbb{H}}\left(x_i,\epsilon\right)}}
\| \left(\nabla_{\mathbb{H}}^{\delta} \varphi\right) \|_{L^{\frac{n}{\gamma}}_{B_{\mathbb{H}}\left(x_i,\epsilon\right)}}
\\
&\ \ \leq \sum_{\alpha+\beta+\gamma+\delta=2l}c_{l,\alpha,\beta,\gamma,\delta}
\|\nabla_{\mathbb{H}}^\alpha f\|_{L^{\frac{2n}{n-2\left(l-\alpha\right)}}_{B_{\mathbb{H}}\left(x_i,\epsilon\right)}}
\|\nabla_{\mathbb{H}}^\gamma f\|_{L^{\frac{2n}{n-2\left(l-\gamma\right)}}_{B_{\mathbb{H}}\left(x_i,\epsilon\right)}}.
\end{split}\end{equation*}
Hence, we have
\begin{equation*}\begin{split}
&\large\left(\int_{B_{\mathbb{H}}\left(x_i,\epsilon\right)}|\varphi|^{\frac{2n}{n-2k}}d\mu \large \right)^{\frac{n-2k}{n}}\cdot C_{n,k}\\
&\ \ \leq\int_{B_{\mathbb{H}}\left(x_i,\epsilon\right)}|\varphi|^2 d\nu+ \sum_{l=1}^{k}\sum_{\alpha+\beta+\gamma+\delta=2l}c_{l,\alpha,\beta,\gamma,\delta}\|\nabla_{\mathbb{H}}^\alpha f\|_{L^{\frac{2n}{n-2\left(l-\alpha\right)}}_{B_{\mathbb{H}}\left(x_i,\epsilon\right)}}
\|\nabla_{\mathbb{H}}^\gamma f\|_{L^{\frac{2n}{n-2\left(l-\gamma\right)}}_{B_{\mathbb{H}}\left(x_i,\epsilon\right)}}.
\end{split}\end{equation*}
Let $\epsilon\to0$, we can obtain that $v^i=\nu\{x_i\}\geq C_{n,k} \left(u^i\right)^{\frac{n-2k}{n}}>0$, $\nu\geq C_{n,k} \left(u^i\right)^{\frac{n-2k}{n}}\delta_{x_i}$ for any $i\in I$.  Therefore,
$$\nu\geq \sum_{i\in I}v^i\delta_{x_i}.$$
By the lower semi-continuity in the sense of measure, we have
$\nu\geq\left( P_k\left(f_0\right)f_0-\prod_{i=1}^{k}\frac{\left(2i-1\right)^2}{4}|f_0|^2\right)dV$.
Thus, we have
$$\nu\geq\left( P_k\left(f_0\right)f_0-\prod_{i=1}^{k}\frac{\left(2i-1\right)^2}{4}|f_0|^2\right)dV+\sum_{i\in I}v^i\delta_{x_i}.$$
We complete the proof of  Lemma \ref{con2}. $\hfill\square$
\vspace{0.8cm}

{\it Proof of Proposition \ref{proconvergence}.}
According to the weak lower-semicontinuity of the measure, we have $\mu\left(\mathbb{B}^n\right)\leq \lim\limits_{m\rightarrow +\infty} \mu_{m}\left(\mathbb{B}^n\right)=1$. From the tightness of $\mu_m$, we know that for any $\epsilon>0$, there exists $R_{\epsilon}$ such that
$$\int_{\mathbb{B}^n\left(0, R_\epsilon\right)}d\mu_m\geq1-\epsilon.$$
Taking $\psi \in C_c^{\infty}\left(\mathbb{B}^n\right)$ with $\psi|_{B^n\left(0, R_\epsilon\right)}=1$ and $0\leq\psi\leq1$,  we have
$$\mu\left(\mathbb{B}^n\right)\geq \int_{\mathbb{B}^n}\psi d\mu\geq\lim\limits_{m\rightarrow +\infty}\int_{\mathbb{B}^n\left(0, R_\epsilon\right)}d\mu_m\geq1-\epsilon.$$
Let $\epsilon\rightarrow 0$, we can obtain $\mu\left(\mathbb{B}^n\right)$=1.

Now we show that $f_0=0$ is impossible. According to the formula \eqref{pkspread}, for $\varphi\in C_c^{\infty}\left(\mathbb{B}^n\right)$ we have
\begin{equation*}\begin{split}
&\lim_{m\rightarrow\infty}\int_{\mathbb{B}^n}P_k\left(f_m\varphi\right)f_m \varphi dV\\
&\ \ =\lim_{m\rightarrow\infty}\int_{\mathbb{B}^n}P_k\left(f_m \right)f_m \varphi^2 dV+\lim_{m\rightarrow\infty}\int_{\mathbb{B}^n} \sum_{l=1}^{k}c_l\sum_{j= 1}^{2l}\nabla_{\mathbb{H}}^{2l-j} f_m \nabla_{\mathbb{H}}^j \varphi \left(f_m \varphi\right) dV \\
\end{split}\end{equation*}
In \eqref{intbypart} and \eqref{intbypart2},  we have argued that $$\lim\limits_{m\rightarrow\infty}\int_{\mathbb{B}^n} \sum_{l=1}^{k}c_l\sum_{j= 1}^{2l}\nabla_{\mathbb{H}}^{2l-j} f_m \nabla_{\mathbb{H}}^j \varphi \left(f_m \varphi\right) dV=\int_{\mathbb{B}^n} \sum_{l=1}^{k}c_l\sum_{j= 1}^{2l}\nabla_{\mathbb{H}}^{2l-j} f_0 \nabla_{\mathbb{H}}^j \varphi \left(f_0 \varphi\right) dV.$$
If $f_0=0$, there holds that
\begin{equation}\label{spread0}
\lim_{m\rightarrow\infty}\int_{\mathbb{B}^n}P_k\left(f_m\varphi\right)f_m \varphi dV\\
=\lim_{m\rightarrow\infty}\int_{\mathbb{B}^n}P_k\left(f_m \right)f_m \varphi^2 dV.
\end{equation}
We can further prove that $\nu\left(\mathbb{B}^n\right)=C_{n,k}$ when $f_0=0$.  For any $\epsilon>0$, choose $\varphi \in C_c^{\infty}\left(\mathbb{B}^n\right)$ with $\varphi|_{B^n\left(0, R_{\epsilon}\right)}=1$, then
\begin{equation*}\begin{split}
\nu\left(\mathbb{B}^n\right)&\geq \int_{\mathbb{B}^n}\varphi^2 d\nu\\
&=\lim_{m\rightarrow\infty}\int_{\mathbb{B}^n}P_k\left(f_m \right)f_m \varphi^2 dV-\prod_{i=1}^{k}\frac{\left(2i-1\right)^2}{4}\int_{\mathbb{B}^n}|f_m \varphi|^2dV \\
&=\lim_{m\rightarrow\infty}\int_{\mathbb{B}^n}P_k\left(f_m\varphi\right)f_m \varphi dV-\prod_{i=1}^{k}\frac{\left(2i-1\right)^2}{4}\int_{\mathbb{B}^n}|f_m \varphi|^2dV\\
&\geq C_{n,k} \left(\int_{\mathbb{B}^n}|f_m\varphi|^{\frac{2n}{n-2k}}dV\right)^{\frac{n-2k}{n}}\\
&\geq C_{n,k} \left(\int_{\mathbb{B}^n\left(0,R_\epsilon\right)}|f_m|^{\frac{2n}{n-2k}}dV\right)^{\frac{n-2k}{n}}\\
&\geq C_{n,k} \left(1-\epsilon\right)^{\frac{n-2k}{n}}.
\end{split}\end{equation*}

Combining  the weak lower-semicontinuity of the measure, $\nu\left(\mathbb{B}^n\right)= C_{n,k}$.

If $f_0=0$, according to the second concentration-compactness principle, $\mu\left(\mathbb{B}^n\right)=1$ and $\nu\left(\mathbb{B}^n\right)=C_{n,k}$, we can derive that
\begin{equation}\begin{split}
1=\mu\left(\mathbb{B}^n\right)&=\sum_{j=1}^{l}\mu_{i}\\
&\leq C_{n,k}^{-\frac{n}{n-2k}}\sum_{j=1}^{l}\left(\nu_{i}\right)^{\frac{n}{n-2k}}\\
&\leq C_{n,k}^{-\frac{n}{n-2k}}\left(\sum_{j=1}^{l}\nu_{i}\right)^{\frac{n}{n-2k}}\\
&=1.
\end{split}\end{equation}
This deduces that $\mu\left(\mathbb{B}^n\right)=\delta_{x_0}$ and $\nu\left(\mathbb{B}^n\right)=C_{n,k}\delta_{x_0}$. Pick $\psi\in C_{c}^{\infty}\left(\mathbb{B}^n\right)$ satisfying $\psi|_{B^n\left(x_0, R\right)}=1$, $supp\left(\psi\right)\subseteq B^n\left(x_0, 2R\right)$ and $0\leq\psi\leq1$, it follows from $\mu_{m}$ converge weakly to $ \mu$  and $\mu_{m}$ converge weakly to  $C_{n,k}\delta_{x_0}$ that
$$\lim\limits_{m\rightarrow +\infty}\int_{\mathbb{B}^n}|f_m\psi|^{\frac{2n}{n-2k}}dV=\lim\limits_{m\rightarrow +\infty}\int_{\mathbb{B}^n}|\psi|^{\frac{2n}{n-2k}}d\mu_m=\int_{\mathbb{B}^n}|\psi|^{\frac{2n}{n-2k}}d\mu=1$$
and
\begin{equation*}
\lim\limits_{m\rightarrow +\infty}\left(\int_{\mathbb{B}^n}P_k\left(f_m\right)f_m|\psi|^2-\prod_{i=1}^{k}\frac{\left(2i-1\right)^2}{4}|f_m \psi|^2dV\right)=C_{n,k}.
\end{equation*}

Combining the above estimates, we obtain
$$\lim\limits_{m\rightarrow +\infty}\frac{\left(\int_{\mathbb{B}^n}P_k\left(f_m\psi\right)\left(f_m\psi\right)dV-\prod_{i=1}^{k}\frac{\left(2i-1\right)^2}{4}\int_{\mathbb{B}^n}|f_m \psi|^2dV\right)}{\left(\int_{\mathbb{B}^n}|f_m\psi|^{\frac{2n}{n-2k}}dV\right)^{\frac{n-2k}{n}}}=C_{n,k}.$$
On the other hand, since $\psi$ has the compact support, $\int_{\mathbb{B}^n}|f_m\psi|^{\frac{2n}{n-2k}}dV\leq 1$, $\psi f_m$ converges to zero for almost $x\in \mathbb{B}^n$, hence it follows from the Vitali convergence theorem that
$$\lim\limits_{m\rightarrow +\infty}\prod_{i=1}^{k}\frac{\left(2i-1\right)^2}{4}\int_{\mathbb{B}^n}|f_m\psi|^{2}dV=0,$$
which implies $$\lim\limits_{m\rightarrow +\infty}\left(\int_{\mathbb{B}^n}P_k\left(f_m\psi\right)\left(f_m\psi\right)dV-\prod_{i=1}^{k}\frac{\left(2i-1\right)^2}{4}\int_{\mathbb{B}^n}|f_m \psi|^2dV\right)=\lim\limits_{m\rightarrow +\infty}\int_{\mathbb{B}^n}P_k\left(f_m\psi\right)\left(f_m\psi\right)dV.$$

Recall the classical Sobolev inequality on the hyperbolic space $\mathbb{B}^n$: for any $u\in C_{c}^{\infty}\left(\mathbb{B}^n\right)$, there holds
$$\int_{\mathbb{B}^n}P_k\left(u\right)u dV\geq S_{n,k}\left(\int_{\mathbb{B}^n}|u|^{\frac{2n}{n-2k}}\right)^{\frac{n-2k}{n}}.$$

Applying the Sobolev inequality into $\psi f_m$, we derive that
\begin{equation}\begin{split}
&\lim\limits_{m\rightarrow +\infty}\frac{\int_{\mathbb{B}^n}P_k\left(f_m\psi\right)\left(f_m\psi\right)dV-\prod_{i=1}^{k}\frac{\left(2i-1\right)^2}{4}\int_{\mathbb{B}^n}|f_m \psi|^2dV}{\left(\int_{\mathbb{B}^n}|f_m\psi|^{\frac{2n}{n-2k}}dV\right)^{\frac{n-2k}{n}}}\\
=&\lim\limits_{m\rightarrow +\infty}\frac{\int_{\mathbb{B}^n}P_k\left(f_m\psi\right)\left(f_m\psi\right)dV}{\left(\int_{\mathbb{B}^n}|f_m\psi|^{\frac{2n}{n-2k}}dV\right)^{\frac{n-2k}{n}}}\\
\geq &S_{n,k},
\end{split}\end{equation}
which is a contradiction with $C_{n,k}$ being strictly smaller than the Sobolev constant $S_{n,k}$. Therefore, $f_0\neq 0$.

Next, According to the Poincar\'e-Sobolev and the Br\'{e}zis-Lieb Lemma in Hilbert space $H\left(\mathbb{B}^n\right)$, we have that
\begin{equation}\begin{split}
1&=\lim\limits_{m\rightarrow +\infty}\int_{\mathbb{B}^n}|f_m|^{\frac{2n}{n-2k}}dV\\
&=\lim\limits_{m\rightarrow +\infty}\int_{\mathbb{B}^n}|f_m-f_0|^{\frac{2n}{n-2k}}dV+\int_{\mathbb{B}^n}|f_0|^{\frac{2n}{n-2k}}dV\\
&\leq \lim\limits_{m\rightarrow +\infty}\left(\frac{1}{C_{n,k}}\right)^{\frac{n}{n-2k}}\left(\|f_m-f_0\|^2_{H}\right)^{\frac{n}{n-2k}}+\left(\frac{1}{C_{n,k}}\right)^{\frac{n}{n-2k}}\left(\|f_0\|^2_{H}\right)^{\frac{n}{n-2k}}\\
&\leq \lim\limits_{m\rightarrow +\infty}\left(\frac{1}{C_{n,k}}\right)^{\frac{n}{n-2k}}\left(\|f_m-f_0\|^2_{H}+\|f_0\|^2_{H}\right)^{\frac{n}{n-2k}}\\
&\leq \lim\limits_{m\rightarrow +\infty}\left(\frac{1}{C_{n,k}}\right)^{\frac{n}{n-2k}}\left(\|f_m\|^2_{H}\right)^{\frac{n}{n-2k}}\\
&=1.
\end{split}\end{equation}
Therefore,  the above all inequalities must be equalities, which together with $f_0\neq 0$ gives that $f_m$ strongly converges to $f_0$ in $H\left(\mathbb{B}^n\right)$. This proves that $f_0$ is the extremal of Poincar\'e-Sobolev inequality on hyperbolic space $\mathbb{B}^n$.
\medskip

Thus we accomplish the proof of Theorem \ref{EPS}.
$\hfill\square$

\section{The proof of Theorem \ref{E-subHPS}}
We now start to prove that the high-order subcritical Poincar\'e-Sobolev inequality on Hyperbolic space $\mathbb{B}^n$:
$$\int_{\mathbb{B}^n}P_k\left(f\right)fdV-\prod_{i=1}^{k}\frac{\left(2i-1\right)^2}{4}\int_{\mathbb{B}^n}|f|^2dV\geq C_{n,k,q}\left(\int_{\mathbb{B}^n}|f|^{q}dV\right)^{\frac{2}{q}}$$
for $2<q<\frac{2n}{n-2k}$
admits an extremal function. Assume that $\{f_m\}$ is a minimizing sequence for the subcritical Poincar\'e-Sobolev inequality on hyperbolic space $\mathbb{B}^n$, that is
$$\int_{\mathbb{B}^n}|f_m|^{q}dV=1,\ \ \lim\limits_{m\rightarrow +\infty}\left(\int_{\mathbb{B}^n}P_k\left(f_m\right)fdV-\prod_{i=1}^{k}\frac{\left(2i-1\right)^2}{4}\int_{\mathbb{B}^n}|f_m|^2dV\right)=C_{n,k,q}.$$
Similar to the proof of the existence of Poincar\'e-Sobolev inequality on hyperbolic space $\mathbb{B}^n$ (Theorem \ref{EPS}), by Green representative formula and Riesz rearrangement inequality, we can assume that the minimizing sequence $\{f_m\}$ is radially decreasing about origin.
In order to prove the existence of extremal function, we only need to prove that
$f_m\rightarrow f_0$ strongly in $L^{q}\left(\mathbb{B}^n\right)$. Since $f_m$ is radially decreasing, by Helly theorem, we know that there exists $f_0\in L^{p}\left(\mathbb{B}^n\right)$ such that $f_m$ converges to $f_0$ for almost $x\in \mathbb{B}^n$. Through Poincar\'e-Sobolev inequality, we see that
$f_m$ is also bounded in $L^{\frac{2n}{n-2k}}\left(\mathbb{B}^n\right)$. This together with the Vitali-convergence theorem yields that
$$\lim\limits_{R\rightarrow 1}\lim\limits_{m\rightarrow +\infty}\int_{B^n\left(0,R\right)}|f_m|^qdV=\lim\limits_{R\rightarrow 1}\int_{B^n\left(0,R\right)}|f_0|^qdV=\int_{\mathbb{B}^n}|f_0|^qdV.$$
In order to prove that $f_m$ strongly converges to $f_0$ in $L^q\left(\mathbb{B}^n\right)$, it remains to show that
$$\lim\limits_{R\rightarrow 1}\lim\limits_{m\rightarrow +\infty}\int_{\mathbb{B}^n\setminus B^n\left(0,R\right)}|f_m|^qdV=0,$$
that is,
$$\lim\limits_{R\rightarrow 1}\lim\limits_{m\rightarrow +\infty}\int_{B^n\setminus B^n\left(0,R\right)}|f_m|^q\left(\frac{2}{1-|x|^2}\right)^ndx=0.$$
From the previous radial Lemma \ref{radial}, we see that there exists $D_{n,k,q_0}>0$ such that
$f_m\left(R\right)\leq \left(1-R\right)^{\frac{n-1}{q_0}}$, where $q_0$ is chosen as $q_0<q$. Then it is not difficult to check that
\begin{equation}\begin{split}
&\lim\limits_{R\rightarrow 1}\lim\limits_{m\rightarrow +\infty}\int_{B^n\setminus B^n\left(0,R\right)}|f_m|^q\left(\frac{2}{1-|x|^2}\right)^ndx\\
\leq &D_{n,k,q_0}\lim\limits_{R\rightarrow 1}\lim\limits_{m\rightarrow +\infty}\int_{B^n\setminus B^n\left(0,R\right)}\left(1-R\right)^{\frac{q\left(n-1\right)}{q_0}}\left(\frac{2}{1-|x|^2}\right)^ndx=0.
\end{split}\end{equation}

Then we accomplish the proof of Theorem \ref{E-subHPS}.

\bibliographystyle{amsalpha}

\end{document}